\documentclass[a4paper,11pt]{article}
\usepackage{amsmath,amssymb,latexsym,amsthm}
\usepackage{array,graphicx,epsfig,xcolor}

\font\elevensf=cmss10 scaled\magstephalf

\newtheorem{theorem} {\bf THEOREM}[section]

\newtheorem{lemma} {\bf LEMMA}[section]

\newtheorem{counter example} {\bf Counter Example}[section]

\newtheorem{remark} {\bf REMARK}[section]

\newtheorem{definition} {\bf DEFINITION}[section]

\title{Exact and optimal controllability for scalar conservation laws with discontinuous flux}

\author{Adimurthi\thanks{Centre for Applicable Mathematics, Tata Institute of Fundamental Research,
Post Bag No 6503, Sharadanagar,
Bangalore - 560065, India. Email: \tt{aditi@math.tifrbng.res.in}.}
\and
Shyam Sundar Ghoshal\thanks{Gran Sasso Science Institute, Viale Francesco Crispi, 7, 67100 L'Aquila, Italy. Email: \tt{shyam.ghoshal@gssi.infn.it}.}
\and  
Pierangelo Marcati\thanks{Gran Sasso Science Institute, Viale Francesco Crispi, 7, 67100 L'Aquila, Italy and 
Department of Information Engineering, Computer Science and Mathematics, University of L'Aquila,
67100 L'Aquila, Italy. Email: \tt{pierangelo.marcati@gssi.infn.it}/\tt{pierangelo.marcati@univaq.it}.}}
\begin{document}
\maketitle
\begin{abstract}
 This paper deals with an optimal control problem and describes the reachable set for the scalar 1-D conservation laws with discontinuous
flux. Regarding the optimal control problem we first prove the existence of a  minimizer and then we prescribe an algorithm  
to compute it. The same method also  applies to compute the initial data control. 
The proof relies on the explicit formula for the conservation laws with the discontinuous flux and  finer properties of the 
characteristics. 
\end{abstract}

\pagestyle{myheadings}
\thispagestyle{plain}
\markboth{Adimurthi, S. S. Ghoshal and P. Marcati}{CONTROLLABILITY FOR  DISCONTINUOUS FLUX}
\section{Introduction} 
\setcounter{equation}{0}
The goal of this paper is to study the optimal and exact control problem of  the following scalar conservation laws  with discontinuous flux 
\begin{equation}\label{conlaw-equation}
 \left\{ \begin{split}
u_t+F(x,u)_x&=0,\\
  u(x,0)&=u_0(x),
 \end{split}
 \right.
\end{equation}
where the flux $F$ is given by, $F(x,u)=H(x)f(u)+(1-H(x))g(u)$, $H$ is the Heaviside function. 
Through out the present article we assume the fluxes $f,g$ to be $C^1(\mathbb{R})$, 
strictly convex with  superlinear growth; $u_0\in L^\infty$.
We denote by $\theta_f,\theta_g$  the unique minima of the fluxes $f,g$ respectively.

There is no literature concerning  reachable set or any sort of optimal controllability results for   equations of type 
(\ref{conlaw-equation}). In the present paper we  obtain a necessary and sufficient condition
for the reachable set
and we prove the existence of a minimizer for an optimal control problem. In order to obtain an initial data control or finding minimizer of 
optimal control, we use a new backward resolution. The advantage of this approach is that it is constructive and easy to compute.
The main difficulty of this backward resolution is that there are no rarefactions originating from the interface $x=0$, then one cannot
just generalize, to the case $f\neq g,$ the backward construction given in \cite{Sco, Sop} for the case in  $f=g$.

In order to state our main results, we need to introduce various notations and technical arguments hence  
the main Theorems \ref{maintheoremop} and  \ref{exact-control} have been postponed to the sections
\ref{section:optimalcontrol} and \ref{section:exactcontrol} respectively.

The scalar conservation laws with discontinuous flux of type (\ref{conlaw-equation}) has a huge variety of applications in several fields, namely
traffic flow modeling,  modeling gravity, modeling continuous sedimentation
in  clarifier-thickener units, ion etching in  the semiconductor industry and many more.

In the past two decades the first order model of type (\ref{conlaw-equation}) has been extensively studied from both the theoretical 
and numerical point of view. Concerning the uniqueness it is worth to mention that the following Kru\v{z}kov
type entropy inequalities in both  the two upper quarter-planes are not sufficient to guarantee the uniqueness, 

\begin{equation}\label{Kruzkov2}
\begin{array}{lll}
 \int\limits_0^\infty \int\limits_0^\infty \left(\phi_1(u)\frac{\partial s}{\partial t}+\psi_1(u)\frac{\partial s}{\partial x}\right)&\geq 
 -\int\limits_0^\infty \psi_1(u(0+,t))s(0,t)dt,\\
  \int\limits_{-\infty}^0 \int\limits_0^\infty \left(\phi_2(u)\frac{\partial s}{\partial t}+\psi_2(u)\frac{\partial s}{\partial x}\right)&\geq 
  \int\limits_0^\infty \psi_2(u(0+,t))s(0,t)dt.
\end{array}
 \end{equation}
Here $(\phi_i,\psi_i)$ denote the entropy pairs  for $i=1,2$ and $s\in C_0^1(\mathbb{R}\times \mathbb{R}_+)$, a non-negative test function. 
Consequently
one need an extra criteria on the interface called "interface entropy condition" (see \cite{Kyoto}) given by 
\begin{equation}\label{entropy-condition}
 \mbox{meas}\{t: f^\prime(u(0+,t))>0, g^\prime(u(0-,t))<0\}=0.
\end{equation}
Using this extra entropy along with the above Kru\v{z}kov type inequalities the 
uniqueness result has been obtained in \cite{Kyoto}.
On the other hand, the existence result has been proved in several ways, namely via Hamilton-Jacobi, 
convergence of numerical schemes, vanishing viscosity method, for further details we refer the reader to  \cite{Kyoto, Siam, Jhde, Jde,  And1, Burger,
 BurKarRisTow, Diehl5, Gimseresebro, Towers} and the references therein. The present paper
uses the explicit formula obtained in  \cite{Kyoto, Jde}, via  the  Hamilton-Jacobi Cauchy problem.
By using this formula it can be shown that if  $v_0$ is uniformly Lipschitz then $v(\cdot, t)$ is also uniformly Lipschitz for all $t>0$ and 
if we denote $u:=\frac{\partial v}{\partial x}$, it follows easily that $u$ is the unique weak solution (see \cite{Kyoto}) satisfying
(\ref{Kruzkov2}),  enjoys (\ref{entropy-condition}) near interface and 
satisfies the following Rankine-Hugoniot condition on the interface.
 \begin{equation}\label{RH-condition}	
   \textrm{meas}\big\{t:f(u(0+,t))\neq g(u(0-,t))\big\}=0.
\end{equation} 
Regarding  the well-posedness theory to $f=g$ case, we refer the reader to \cite{Da1} for Cauchy problem 
and for the initial boundary value problem to \cite{Jos}.

Through out this paper we work with the solution which is obtained from the Hamilton-Jacobi formulation.
 
Concerning the exact controllability for the scalar convex conservation laws the first work has been done in \cite{An1}, where they considered
the initial boundary value problem in a quarter plane with $u_0=0$ and by using  one boundary control they investigated  the reachable set. 
As in \cite{Sco}, they considered  $u_0\in L^\infty$ and three possible cases, namely pure initial value problem with 
initial data control outside any domain, initial boundary value problem in a quarter plane with one boundary control and initial boundary problem in a strip with
two boundary controls to get the reachable sets in a complete generality. In both the articles the Lax-Oleinik type formulas has been exploited.
An alternative approach has been provided in \cite{Hor} by using the return method (see \cite{Coron}). For the viscous Burgers equation any non-zero
state can be reached in finite time by two boundary controls \cite{Glass}.
A general theory for the system of conservation
laws is still largely unavailable, nevertheless
in \cite{Bre1}, the authors  constructed an example showing that exact controllability to a constant cannot be reached  in a finite time and proved
asymptotic stabilization to a constant  by two boundary controls. Recently, under dissipative boundary conditions the 
asymptotic stabilization to $0$ has been proved in \cite{CoronShyam} for $2\times 2$ system, when the velocities are positive.
For the Temple class systems and triangular type systems we refer the reader 
to \cite{An2} and \cite{SamJEE} respectively. 

Let us briefly discuss the optimal controllability results for the case $f=g$.  Assume  the target function
$k\in L^2_{loc}$, $T>0$, we denote by  $J_{\{f=g\}}$, a cost functional, defined in the following way 
\begin{equation}\begin{array}{ll}
                  J_{\{f=g\}}(u_0)=\int\limits_{-\infty}^{\infty}|f^\prime(u(x,T))- f^\prime(k(x))|^2dx ,
                \end{array}                         
\label{optimalcontrolf=g}\end{equation}
where $u_0\in L^\infty(\mathbb{R})$, $u_0\equiv\theta_f$ outside a compact set, $\theta_f$  being the only critical point of the flux $f$.
Here $u(\cdot,T)$ denotes the unique weak solution  at $t=T$ to the Cauchy problem 
 (\ref{conlaw-equation}), in the case  $f=g$ with  initial datum $u_0$. Then in this case, 
the optimal control reads like : find a $w_0$ such that $J_{\{f=g\}}(w_0)=\min\limits_{u_0}J_{\{f=g\}}(u_0).$
 In \cite{Cas, Cas2}, they considered the above optimal control problem for the  Burgers' equation and proved such  minimizer exists and proposed a 
numerical scheme called "alternating decent algorithm", although the convergence of these scheme  still remains open.
Whereas in \cite{Sop}, they made use of the Lax-Oleinik formula and derived a numerical backward construction which 
converges to a solution  of the above problem. The latter method can be applied also to  general convex fluxes as long as a
Lax-Oleinik type formula is available. It has to be noticed that even for the case $f=g$, 
due to the occurrence of the shocks in the solution of (\ref{conlaw-equation}), one may have several minimizers of the optimal control problem
(\ref{optimalcontrolf=g}).
 
 We organize the paper in the following manner, section 2 deals with the existing results collected from \cite{Kyoto, Jde}. Section 3 consists
 of some important Lemmas and the backward construction. Then we state and prove the result concerning optimal controllability in section 4. Finally
 in section 5, we state and prove the exact controllability result.

\section{Known facts  about discontinuous fluxes}\label{section:knownresults}
In order to make the paper self contained we recall  some results, definitions and notations  from \cite{Kyoto}.

\begin{definition} \textbf{Control curve} : Let $0 <t,\ x\in \mathbb{R} $ and $\gamma\in C([0,t],\mathbb{R})$. Then $\gamma$ is said to be a 
control curve if the following holds: it is piecewise affine on $[0,t]$
with at most 3 segments, each segment must be completely inside a closed quarter plane. If they are exactly 3, then the middle one is on the 
line $x=0$ and the other two must be either in the positive or negative quarter plane. Moreover, no segment can cross $x=0$. Let $c(x,t)$ be the 
set of control curves, $c_0(x,t)$ is the subset of $c(x,t)$ consists of only one segment. $c_r(x,t)$ is the subset containing exactly 3 segments
and     $c_b(x,t)= c(x,t)\setminus \{c_0(x,t) \cup c_r(x,t)\}$.
\end{definition}
\begin{definition}\textbf{Cost function}: Let  $f^*, g^*$ be the convex duals of the fluxes. 
Let us assume that $v_0: \mathbb{R}\rightarrow \mathbb{R}$ be an  uniformly Lipschitz continuous function. 
Let $(x,t)\in\mathbb{R}\times \mathbb{R}_+$,  $\gamma\in c(x,t)$.
The cost functional $\Gamma$ associated to $v_0$ is defined by 
\begin{eqnarray*}\begin{array}{llll}
 \Gamma_{v_0,\gamma}(x,t)=v_0(\gamma(0))+\int\limits_{\{\theta\in[0,t]\ :\ \gamma(\theta)>0 \}} f^*(\dot{\gamma}) d\theta
 +\int\limits_{\{\theta\in[0,t]\ :\ \gamma(\theta)<0 \}} g^*(\dot{\gamma}) d\theta \\+ \mbox{meas}\{\theta\in[0,t]\
 :\ \gamma(\theta)=0 \}\mbox{min}\{f^*(0),g^*(0)\}.
              \end{array}
\end{eqnarray*}
Then we define the value function   $v:\mathbb{R}\times \mathbb{R}_+ \rightarrow \mathbb{R}$ by 
 $v(x,t)=\inf\limits_{\gamma\in c(x,t)}\{ \Gamma_{\gamma, v_0}(x,t) \}.$
\end{definition}
\begin{definition} Let us define by 
 $ch(x,t)=\{\gamma \ : \ \Gamma_{v_0,\gamma}(x,t)=v(x,t)\},$ the set of \textbf{characteristics curves}.
Let $t>0,$ define 
\begin{eqnarray*}\begin{array}{lllll}
R_1(t)&=&\inf\{x\ ;\ x\geq0,\ ch(x,t)\subset c_0(x,t)\},\\ 
R_2(t)&=&\left\{\begin{array}{lll}\inf\{x\ ;\ 0\leq x\leq R_1(t),\ ch(x,t)\cap c_r(x,t)\neq\phi\},\\ 
                 R_1(t)\quad \mbox{if the above set is empty}.
                \end{array}\right.\\
L_1(t)&=&\sup\{x\ ;\ x\leq0,\ ch(x,t)\subset c_0(x,t)\},\\ 
L_2(t)&=&\left\{\begin{array}{lll}\sup\{x\ ;\ L_1(t)\leq x\leq 0,\ ch(x,t)\cap c_r(x,t)\neq\phi\},\\ 
                 L_1(t)\quad \mbox{if the above set is empty}.
                \end{array}\right.\end{array}
                                 \end{eqnarray*}\end{definition}
                    For   $f=g$ case, a detailed study of the above curves has been done in \cite{Sco, Ssh}.
\begin{theorem}\label{AG1} (See \cite{Kyoto}) Let $u_0\in L^\infty(\mathbb{R})$ and  $v_0(x):= \int\limits_0^x u_0(\theta) d\theta.$
Then 
\begin{enumerate}
  \item Then the   function $v$ is uniformly Lipschitz continuous and  $u:=\frac{\partial v}{\partial x}$
  solves (\ref{conlaw-equation}) in weak sense with initial data $u_0$.
 \item $u$ satisfies Rankine-Hugoniot condition (\ref{RH-condition}) and interface entropy condition  (\ref{entropy-condition}) near the interface.
\item $R_1(\cdot), L_1(\cdot)$ are Lipschitz continuous functions and there exists a function  
$y: \{(-\infty,L_1(t)]\cup [R_1(t),\infty)\}\times (0,\infty)\rightarrow \mathbb{R}$  such that for all $t>0$, $x\mapsto y(\cdot,t)$ 
is non decreasing function.
\item  There exist   non increasing function $t_+:[0,R_1(t)]\rightarrow [0,t]$ and a  non decreasing function $t_-:[L_1(t),0]\rightarrow [0,t]$ 
such the Explicit formula is given by  
\begin{eqnarray}
 f^\prime(u(x,t))&=
 \left(\frac{x-y(x,t)}{t}\right)\mathbf{1}_{x\geq R_1(t)}+ 
\left(\frac{x}{t-t_+(x,t)}\right) \mathbf{1}_{0\leq x<R_1(t)}.
\label{r2.12}\\
 g^\prime(u(x,t))&=
 \left(\frac{x-y(x,t)}{t}\right)\mathbf{1}_{x\leq L_1(t)}+ \left(\frac{x}{t-t_-(x,t)}\right) \mathbf{1}_{L_1(t)<x<0}.
\label{r2.13}
\end{eqnarray}
\item Let $V_+=\{t: R_1(t)>0\}, V_-=\{ t: L_1(t)<0 \}$. Then there exist non increasing function $y_{-,0}:V_+\rightarrow (-\infty,0]$ and a
non decreasing function $y_{+,0}:V_-\rightarrow [0,\infty)$ such that 
$
 g^\prime(u(0-,t))=- \frac{y_{-,0}(t)}{t},$ for $t\in V_+\setminus D_+,$
 $f^\prime(u(0+,t))=- \frac{y_{+,0}(t)}{t},$ for $t\in V_-\setminus D_-$
and $f(u(0+,t))=g(u(0-,t)),$ for $t\in (V_+\setminus D_+)\cup (V_-\setminus D_-),
$
where $D_\pm$ are the points of discontinuities of $y_{\pm,0}.$
\item For each $T>0,$ one of the following holds, 
\begin{enumerate}
 \item [i).] If $R_1(T)>0, L_1(T)=0$, then $\forall\ t\in (t_+(R_1(T)-,T), T), R_1(t)>0,$\\
 where  $t_+(R_1(T)-,T)=\lim\limits_{x\rightarrow R_1(T)-}t_+(x,T),$ (see figure \ref{Fig1}, case i).
  \item [ii).] If $R_1(T)=0, L_1(T)<0$, then $\forall\ t\in (t_-(L_1(T)+,T), T), L_1(t)<0,$\\
  where $t_-(L_1(T)+,T)=\lim\limits_{x\rightarrow L_1(T)+}t_-(x,T)$.
 \item [iii).] $R_1(T)=0=L_1(T),$ (see figure \ref{Fig1}, case iii).
\end{enumerate}\end{enumerate}\end{theorem}
\begin{figure}[ht]\label{Fig1}
        \centering
        \def\svgwidth{0.8\textwidth}
        \begingroup
    \setlength{\unitlength}{\svgwidth}
  \begin{picture}(1,0.55363752)%
    \put(0,0.1){\includegraphics[width=0.85\textwidth]{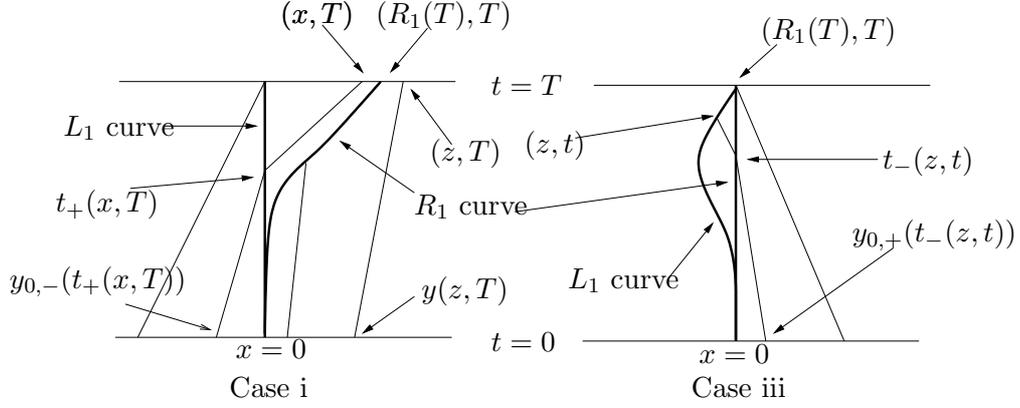}}%
    \put(0.39,0.27){\color[rgb]{0,0,0}\makebox(0,0)[lb]{\smash{$R_1$ curve}}}%
        \put(0.999,0.33){\color[rgb]{0,0,0}\makebox(0,0)[lb]{\smash{$t_-(z,t)$}}}%
                \put(0.959,0.234){\color[rgb]{0,0,0}\makebox(0,0)[lb]{\smash{$y_{0,+}(t_-(z,t))$}}}%
                \put(0.533,0.35){\color[rgb]{0,0,0}\makebox(0,0)[lb]{\smash{$(z,t)$}}}%
    \put(-0.076,0.273){\color[rgb]{0,0,0}\makebox(0,0)[lb]{\smash{$t_+(x,T)$}}}%
        \put(-0.136,0.173){\color[rgb]{0,0,0}\makebox(0,0)[lb]{\smash{$y_{0,-}(t_+(x,T))$}}}%
        \put(.59,0.173){\color[rgb]{0,0,0}\makebox(0,0)[lb]{\smash{$L_1$ curve}}}%
                \put(-0.066,0.373){\color[rgb]{0,0,0}\makebox(0,0)[lb]{\smash{$L_1$ curve}}}%
    \put(0.49,0.09){\color[rgb]{0,0,0}\makebox(0,0)[lb]{\smash{$t=0$}}}%
        \put(0.76,0.075){\color[rgb]{0,0,0}\makebox(0,0)[lb]{\smash{$x=0$}}}%
                \put(0.216,0.52){\color[rgb]{0,0,0}\makebox(0,0)[lb]{\smash{$(x,T)$}}}%
                                \put(0.216,0.52){\color[rgb]{0,0,0}\makebox(0,0)[lb]{\smash{$(x,T)$}}}%
                                \put(0.34,0.52){\color[rgb]{0,0,0}\makebox(0,0)[lb]{\smash{$(R_1(T),T)$}}}%
                                \put(0.84,0.5){\color[rgb]{0,0,0}\makebox(0,0)[lb]{\smash{$(R_1(T),T)$}}}%
                                                                \put(0.41,0.34){\color[rgb]{0,0,0}\makebox(0,0)[lb]{\smash{$(z,T)$}}}%
        \put(0.158,0.081){\color[rgb]{0,0,0}\makebox(0,0)[lb]{\smash{$x=0$}}}%
                \put(.15,0.03){\color[rgb]{0,0,0}\makebox(0,0)[lb]{\smash{Case i}}}%
                  \put(.75,0.03){\color[rgb]{0,0,0}\makebox(0,0)[lb]{\smash{Case iii}}}%
                   \put(0.399,0.16){\color[rgb]{0,0,0}\makebox(0,0)[lb]{\smash{$y(z,T)$}}}%
    \put(0.49,0.4269){\color[rgb]{0,0,0}\makebox(0,0)[lb]{\smash{$t=T$}}}%
  \end{picture}%
\endgroup
        \caption{Case i : when $R_1(T)>0, L_1(T)=0$, Case iii : when $R_1(T)=0, L_1(T)=0.$}
        \label{fig:2}
\end{figure}
\section{Key technical lemmas}
Let us denote
$ f_+=f\mid_{[\theta_f,\infty)}, f_-=f\mid_{(-\infty,\theta_f]}, g_+=g\mid_{[\theta_g,\infty)}$, $
g_-=g\mid_{(-\infty,\theta_f]}$, where $\theta_f,\theta_g$ be the unique minimums of the 
fluxes $f,g$ respectively.
Let  $\bar{\theta}_f, \bar{\theta}_g$  such that 
\begin{eqnarray*}
\left\{\begin{array}{lll}
f^\prime(\bar{\theta}_g)\geq 0, f(\bar{\theta}_g)=g(\theta_g) &\mbox{if}& g(\theta_g)\geq f(\theta_f),\\
g^\prime(\bar{\theta}_f)\leq 0, f(\theta_f)=g(\bar{\theta}_f) &\mbox{if}& g(\theta_g)\leq f(\theta_f).
\end{array}\right.
\end{eqnarray*}
Let 
\begin{eqnarray*}
I_+=\left\{\begin{array}{lllll}
\ [\bar{\theta}_g,\infty) &\mbox{if}& g(\theta_g)\geq f(\theta_f),\\
\ [\theta_f,\infty) &\mbox{if}& g(\theta_g)\leq f(\theta_f).
\end{array}\right.
\end{eqnarray*}
Define $h_+:I_+\rightarrow [0,\infty)$ by 
$ h_+(p)=g^\prime\circ g^{-1}_+\circ f\circ {f^{\prime}}^{-1}(p).$
Then clearly $h_+$ is well defined and strictly increasing function. Let 
\begin{eqnarray*}
I_-=\left\{\begin{array}{lllll}
\ [\theta_g,\infty) &\mbox{if}& g(\theta_g)\geq f(\theta_f)\\
\   [\bar{\theta}_f,\infty) &\mbox{if}& g(\theta_g)\leq f(\theta_f).
\end{array}\right.
\end{eqnarray*}
Similarly we define  a  strictly decreasing function $h_-:I_-\rightarrow (-\infty,0]$ by 
$ h_-(p)=f^\prime\circ f^{-1}_-\circ g\circ {g^{\prime}}^{-1}(p).$
\subsection{No rarefaction waves on the interface}
The following Lemma plays a key role in our result since excludes that  forward rarefaction waves  emanates  from the interface.
\begin{lemma}\label{norarefaction}
 $x\mapsto t_+(x,t)$ is a  strictly decreasing function in $(0,R_1(t))$ and 
  $x\mapsto t_-(x,t)$ is  a strictly increasing function in $(L_1(t),0)$.
\end{lemma}
\begin{proof}
 Without loss of generality we assume that for some $T>0$,  $R_1(T)>0$ and $L_1(T)=0.$
 Define $X_T:[t_+(R_1(T)-,T), T]\rightarrow [0,R_1(T)]$ by $X_T(t)=\mbox{Max}\{x\ :\ t_+(x,T)\geq t\}.$ Since $x\mapsto t_+(x,T)$
is a non increasing function, hence $t\mapsto X_T(t)$ is a  non increasing function. Since the limits of characteristics  curves are 
characteristics curves,
hence there exists a $\gamma\in (c_r(X_T(t),T)\cup c_b(X_T(t),T))\cap ch(X_T(t),T)$ such that $\gamma=(\gamma_1,\gamma_2,\gamma_3),$
where the first component of $\gamma$ is given by 
 $
 \gamma_1(\theta)=X_T(t)+(\theta-T)\frac{X_T(t)}{T-t_+(X_T(t)-,t)}, \ \  \mbox{for}\ \theta\in [t_+(X_T(t)-,T),T].
$
Let $D_1=$ set of discontinuities of the mapping $t\mapsto X_T(t)$. Then $D_1$ is countable and as in Step 1, Lemma 4.10 of \cite{Kyoto}, it 
can be shown that for all $t\notin D_1.$ The R-H condition holds, i.e,. 
$f(u(0+,t))=g(u(0-,t))\ \mbox{for all}\ t\notin D_1 $. Now we state the following claim which will conclude the Lemma. 

\noindent\textbf{Claim 1}: $t\mapsto X_T(t)$ is a continuous function. 

Suppose not, then there exists $t_0\in [0,T]$ such that
\begin{equation}\label{x0-x3}
 x_0=X_T(t_0+)<X_T(t_0-)=x_1.
\end{equation}
Due to the fact that characteristics do not intersect properly, we have  for all $x\in(x_0,x_1)$, $t_+(x,T)=t_0.$ 
Observe that the straight line $\alpha(t):=(t-t_0)\left(\frac{x_1}{T-t_0}\right)$ is a characteristic curve. 
Since characteristics do not intersect properly, we conclude
$R_1(t)>0, \ \mbox{for all} \ t\in (t_0,T],$ hence by entropy condition (\ref{entropy-condition}), we have 
$ L_1(t)=0\ \mbox{for all}\ t\in(t_0,T].$

\noindent\textbf{Claim 2}: There exists $\epsilon>0$ such that for all $t\in[t_0-\epsilon,t_0]$, $R_1(t)>0$.\\
If not, then there exists a sequence $\{t_k\}_{k=1}^\infty$ with $t_{k+1}>t_k$ such that $R_1(t_k)=0$ and
$\lim\limits_{k\rightarrow \infty} t_k=t_0.$ Since $R_1(t_k)=0$, therefore there exists a
sequence $\{y_k\}_{k=1}^\infty$ with $y_k\geq 0$ such that 
$
\beta_k(t):=(t-t_k)(-\frac{y_k}{t_k}) 
$
is a characteristic curve. 
The function $t\mapsto y_{+,0}(t)$ is non decreasing hence we conclude $y_{k+1}\geq y_k$, for all $k$. Then  $\{y_k\}_{k=1}^\infty$
is a non decreasing sequence and bounded below by $0$, therefore converges to some $y_0\geq 0$ (say). 
By the definition of characteristics curve we conclude
\begin{equation}\label{charseq}
 v(0,t_k)=v_0(y_k)+t_kf^*\left(-\frac{y_k}{t_k}\right).
\end{equation}
Since $v$ is uniformly Lipschitz and $y_k,t_k$ converges to $y_0,t_0$ respectively, we pass to the limit in the equation (\ref{charseq})
 to obtain $
   v(0,t_0)=v_0(y_0)+t_0f^*\left(-\frac{y_0}{t_0}\right).$
 Which proves that the straight line $\beta(t):=(t-t_0)(-\frac{y_0}{t_0})$, is a characteristic curve.
Define the straight line $\gamma(t):=x_1+(t-T)\left(\frac{\alpha(T)-\beta(0)}{T}\right)$.
By using the fact that $\alpha,\beta$ are characteristics curves and  $f^*$ is a strict convex function, we obtain
\begin{eqnarray*}
 \begin{array}{llll}
  v(\alpha(T),T)&\leq& v_0(\gamma(0))+Tf^*\left(\frac{\alpha(T)-\beta(0)}{T}\right)
  = v_0(\beta(0))+Tf^*\left(\frac{\alpha(T)-\beta(0)}{T}\right)\\
  &<&v_0(\beta(0)) +t_0f^*(\dot{\beta})+(T-t_0)f^*\left(\dot{\alpha}\right)
  = v(0,t_0)+(T-t_0)f^*\left(\dot{\alpha}\right)\\
  &=&v(\alpha(T),T),
 \end{array}
\end{eqnarray*}
 which is a contradiction. Hence the claim 2. 
\begin{figure}[ht]\label{Fig2}
        \centering
        \def\svgwidth{0.8\textwidth}
        \begingroup
    \setlength{\unitlength}{\svgwidth}
  \begin{picture}(1,0.55363752)%
    \put(0,0.1){\includegraphics[width=0.7\textwidth]{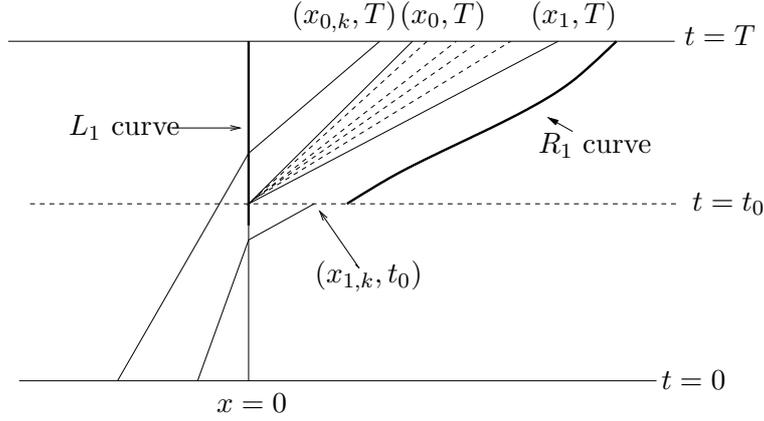}}%
    \put(0.69,0.4){\color[rgb]{0,0,0}\makebox(0,0)[lb]{\smash{$R_1$ curve}}}%
    \put(0.08,0.423){\color[rgb]{0,0,0}\makebox(0,0)[lb]{\smash{$L_1$ curve}}}%
    \put(0.85,0.09){\color[rgb]{0,0,0}\makebox(0,0)[lb]{\smash{$t=0$}}}%
        \put(0.89,0.3249){\color[rgb]{0,0,0}\makebox(0,0)[lb]{\smash{$t=t_0$}}}%
        \put(0.272,0.06){\color[rgb]{0,0,0}\makebox(0,0)[lb]{\smash{$x=0$}}}%
    \put(0.88,0.5409){\color[rgb]{0,0,0}\makebox(0,0)[lb]{\smash{$t=T$}}}%
                \put(0.37,0.57){\color[rgb]{0,0,0}\makebox(0,0)[lb]{\smash{$(x_{0,k},T)$}}}%
            \put(0.51,0.57){\color[rgb]{0,0,0}\makebox(0,0)[lb]{\smash{$(x_0,T)$}}}%
                        \put(0.68,0.57){\color[rgb]{0,0,0}\makebox(0,0)[lb]{\smash{$(x_1,T)$}}}%
    \put(0.4,0.23){\color[rgb]{0,0,0}\makebox(0,0)[lb]{\smash{$(x_{1,k},t_0)$}}}%
  \end{picture}%
\endgroup
        \caption{An illustration of the Lemma \ref{norarefaction}}
        \label{fig:qcsets}
\end{figure}
 Therefore by entropy condition (\ref{entropy-condition}) and claim 2, we obtain 
\begin{equation}\label{L1=0}
 \mbox{for all}\  t\in[t_0-\epsilon,T], L_1(t)=0 \ \mbox{and}\ R_1(t)>0.
\end{equation}
Due to R-H condition, (\ref{x0-x3}) and  (\ref{L1=0}),  we can consider the sequences
$\{x_{0,k}\}_{k=1}^\infty$, $\{x_{1,k}\}_{k=1}^\infty,$
$ \{t_k\}_{k=1}^\infty$, $\{\bar{t}_k\}_{k=1}^\infty$ with $\lim\limits_{k\rightarrow \infty} x_{0,k}=x_0,$ $\lim\limits_{k\rightarrow \infty} x_{1,k}=0,$
$\lim\limits_{k\rightarrow \infty} t_k=t_0, \lim\limits_{k\rightarrow \infty} \bar{t}_k=t_0$ such that for all $k\in\mathbb{N},$
 $x_{0,k+1}\geq x_{0,k}, x_{1,k+1}<x_{1,k}$, $t_{k+1}\leq t_k, \bar{t}_{k+1}>\bar{t}_k$, 
  $f(u(0+,t_k))=g(u(0-,t_k))$,   $f(u(0+,\bar{t}_k)) = g(u(0-,\bar{t}_k)),$
                    $\eta_k(t):=(t-t_k)\left(\frac{x_{0,k}}{T-t_k}\right)$ and 
          $\bar{\eta}_k(t):=(t-\bar{t}_k)\left(\frac{x_{1,k}}{t_0-\bar{t}_k}\right)$ are characteristics curves.
                                        Note that the slopes characteristics curves $\eta_k$ and $\bar{\eta}_k$
 converges to the slopes  characteristics curves $\eta(t):=(t-t_0)\left(\frac{x_{0}}{T-t_0}\right)\ 
 \mbox{and}\ \bar{\eta}(t):=(t-{t}_0)\left(\frac{x_{1}}{T-{t}_0}\right)$ respectively, which proves 
 \begin{equation}\label{limu0}
  \lim\limits_{k\rightarrow \infty} u(0+,t_k)=u(0+,t_0+)\mbox{(say)}\ \mbox{and}\  \lim\limits_{k\rightarrow \infty} u(0+,\bar{t}_k)=u(0+,t_0-)\mbox{(say)}.
 \end{equation}
On the other hand, $\eta(t)$ and  $\bar{\eta}(t)$ are characteristics curves, hence 
 $u(0+,t_0+)=(f^\prime)^{-1}(\frac{x_0}{T-t_0})$ and $u(0+,t_0-)=(f^\prime)^{-1}(\frac{x_1}{T-t_0})$. 
 Since $x_0<x_1,$ clearly 
\begin{equation}\label{ineu0}
 u(0+,t_0+)<u(0+,t_0-).
\end{equation}
Therefore from (\ref{limu0}) and (\ref{ineu0}), there exists a $\delta_1>0$ and $m\in\mathbb{N}$ such that for all $k>m,$
\begin{equation}\label{delta1ine}
 u(0+,\bar{t}_k)-u(0+,t_k)>\delta_1.
\end{equation}
Again by R-H condition (\ref{RH-condition}) and (\ref{delta1ine}), there exists a $\delta_2>0$ such that for all $k>m,$
  \begin{equation}\label{delta2ine}
 g^\prime(u(0+,\bar{t}_k))-g^\prime(u(0+,t_k))>\delta_2.
\end{equation}
Due to the fact that $L_1(t)=0$ in the neighborhood of $t_0$, by using explicit formulas there exists sequences  $\{y_{k}\}_{k=1}^\infty, 
\{\bar{y}_{k}\}_{k=1}^\infty$ such that $g^\prime (u(0-,t_k))=-\frac{y_k}{t_k}$ and $g^\prime (u(0-,\bar{t}_k))=-\frac{\bar{y}_k}{\bar{t}_k}$.
The function $t\mapsto y_{-,0}(t)$ is non increasing hence we conclude $y_{k+1}\geq y_k$, $\bar{y}_{k+1}\leq \bar{y}_k$,for all $k$. 
Since the sequences $\{y_{k}\}_{k=1}^\infty, 
\{\bar{y}_{k}\}_{k=1}^\infty$ are monotonic, bounded and due to the fact that  characteristics do not intersect properly, we conclude
 \begin{equation}\label{y0<y0}
  \lim\limits_{k\rightarrow \infty} y_k =y_0 \mbox{(say)}\leq  \lim\limits_{k\rightarrow \infty} \bar{y}_k=\bar{y}_0\mbox{(say)}.
 \end{equation}
Exploiting the explicit formula we obtain 
\begin{eqnarray}\label{diffex}
 \begin{array}{llllllllll}
g^\prime(u(0+,\bar{t}_k))-g^\prime(u(0+,t_k))
&=&\displaystyle \frac{-\bar{y}_kt_k+y_k\bar{t}_k}{\bar{t}_kt_k}.
 \end{array}
\end{eqnarray}
As $\lim\limits_{k\rightarrow \infty}t_k=t_0, \lim\limits_{k\rightarrow \infty}\bar{t}_k=t_0$ and (\ref{y0<y0}), 
the right hand side of (\ref{diffex}) converges to some non positive number but due to (\ref{delta2ine}) the left hand side of (\ref{diffex})
remain strictly positive, which is a contradiction. This proves claim 1. Therefore $x\mapsto t_+(x,t)$ is strictly decreasing function in  $(0,R_1(t))$.
Similarly one can prove that $x\mapsto t_-(x,t)$ is a strictly increasing function in  $(L_1(t),0)$. Hence the Lemma. 
\end{proof}
\begin{remark}
 There are no rarefaction start from the interface at any positive time.
\end{remark}
\subsection{Explicit formulas connecting the interface} The following two lemmas explains how the solution of (\ref{conlaw-equation}) at time $t=T$ 
connected to $t=0$ via  characteristics through the interface.
\begin{lemma}
 Let $T>0$ and denote $t_\pm(x,T)=t_\pm(x).$ Then 
 \begin{enumerate}
  \item 
 For a.e.  $x\in[0,R_1(T))$, we have  $\displaystyle-\frac{y_{-,0}(t_+(x))}{t_+(x)}=h_+\left(\frac{x}{T-t_+(x)}\right).$
  \item For a.e. $x\in(L_1(T),0]$, we have $\displaystyle  - \frac{y_{+,0}(t_-(x))}{t_-(x)}=h_-\left(\frac{x}{T-t_-(x)}\right).$
    \end{enumerate}
\end{lemma}
\begin{proof}
 It is enough to prove (1), (2)  follows in the same direction. Let $R_1(T)>0$. Then from (6) of Theorem \ref{AG1}, $R(t)>0$ for
 $t\in (t_+(R(T))-, T)$ and hence from (5) of Theorem \ref{AG1}, $y_{-,0}$ is well defined on $(t_+(R(T)),T).$ Again from (2) of Theorem \ref{AG1},
 $t_+$ is a strictly decreasing function, hence the set 
\begin{equation*}
 E_+=\{t_+^{-1}(D_+)\}\cup \{\mbox{points of discontinuities of } t_+ \}
\end{equation*}
is a countable set. Now from (5) of  Theorem \ref{AG1}, if $x\notin E_+,$ then $t_+(x)\notin D_+$ and hence 
\begin{eqnarray}
 f(u(0+,t_+(x)))=g(u(0-,t_+(x))),\ g^\prime(u(0-,t_+(x)))=-\frac{y_{-,0}(t_+(x))}{t_+(x)}.\label{i}
\end{eqnarray}
From (4) and (5) of   Theorem \ref{AG1}, we have $f^\prime(u(x,T))=\frac{x}{T-t_+(x)}$ for a.e. $x\in [0,R_1(T)]$. 
This implies at the point of continuity of $t_+$, we have 
\begin{equation}
 f^\prime(u(x,T))=f^\prime(u(0+,t_+(x))).\label{iv}
\end{equation}
Therefore from \eqref{i}-\eqref{iv}, for $x\notin E_+$, we  conclude the proof of (1).
\end{proof}

\begin{lemma}\label{simplelemma}
 Let $\rho:[\alpha,\beta]\subset (0,\infty) \rightarrow (-\infty,0)$ be a non decreasing function.
\begin{enumerate}
 \item  Let $t: [\alpha,\beta] \rightarrow [0,T]$ be a function 
 such that 
 \begin{equation}\label{RH1}
  -\frac{\rho(x)}{t(x)}=h_+\left(\frac{x}{T-t(x)}\right) \ \mbox{a.e.}\ x\in [\alpha,\beta].
 \end{equation}
Then
 $x\mapsto t(x)$ is a strictly decreasing function.
 \item  For $i=1,2$, let  $t_i: [\alpha,\beta] \rightarrow [0,T]$ be two functions
 such that 
 \begin{equation}\label{RH1-2}
  -\frac{\rho(x)}{t_i(x)}=h_+\left(\frac{x}{T-t_i(x)}\right) \ \mbox{a.e.}\ x\in [\alpha,\beta].
 \end{equation}
 Then $t_1(x)=t_2(x), \ \mbox{a.e.}\ x\in [\alpha,\beta].$
\end{enumerate}
\begin{proof}
 Let $0<x_1<x_2$ and (\ref{RH1}) holds at $x_1$ and $x_2$. Suppose $t(x_1)\leq t(x_2).$ Then
 $\frac{x_1}{T-t(x_1)}\leq \frac{x_1}{T-t(x_2)}<\frac{x_2}{T-t(x_2)}$. Hence 
 $$-\rho(x_1)=t(x_1)h\left(\frac{x_1}{T-t(x_1)}\right)<t(x_2)h\left(\frac{x_2}{T-t(x_2)}\right)=-\rho(x_2)\leq -\rho(x_1),$$
which is a contradiction. This proves (1). Proof of (2) is immediate.
 \end{proof}
 \end{lemma}
\subsection{Backward wave analysis} The following lemmas proves the existence of possible functions $t_+, u_0$, given $\rho$.
\begin{lemma}\label{discontinuouslemma}
 Let $x_0>0, T>t_1>t_2>0$. Define $\rho_1,\rho_2\in \mathbb{R}$ such that 
 \begin{eqnarray}
  -\frac{\rho_i}{t_i}=h_+\left(\frac{x_0}{T-t_i}\right).
 \end{eqnarray}
Suppose $\rho_1<\rho_2<0$, then there exists a  solution $u\in L^\infty(\mathbb{R}\times [0,T]) $ for (\ref{conlaw-equation}). \end{lemma}

 \begin{proof}
 Let us denote $a_1,a_2$, such that $\frac{x_0}{T-t_i}=f^\prime(a_i)$, for $i=1,2.$ Then by strict convexity of $f$, one obtains $a_1>a_2.$ 
 Let us denote $b_1,b_2$, such that $b_i=g_+^{-1}f_+(a_1)$, for $i=1,2.$ Consider the line $x-x_0=\frac{f(a_1)-f(a_2)}{a_1-a_2}(t-T)$, this 
 line hits the $x=0$ at time $t=T-x_0s_2(t-T)=t_3$ (say), where $s_2=\frac{f(a_1)-f(a_2)}{a_1-a_2}$. Again by strict convexity $t_1>t_3>t_2.$
 Let us define the initial data $u_0$, by 
 $
 u_0(x)= b_1\mathbf{1}_{x<\rho_3}+b_2\mathbf{1}_{x>\rho_3}.$
\begin{figure}[ht]\label{Fig3}
        \centering
        \def\svgwidth{0.8\textwidth}
        \begingroup
    \setlength{\unitlength}{\svgwidth}
  \begin{picture}(1,0.55363752)%
    \put(0,0.1){\includegraphics[width=0.7\textwidth]{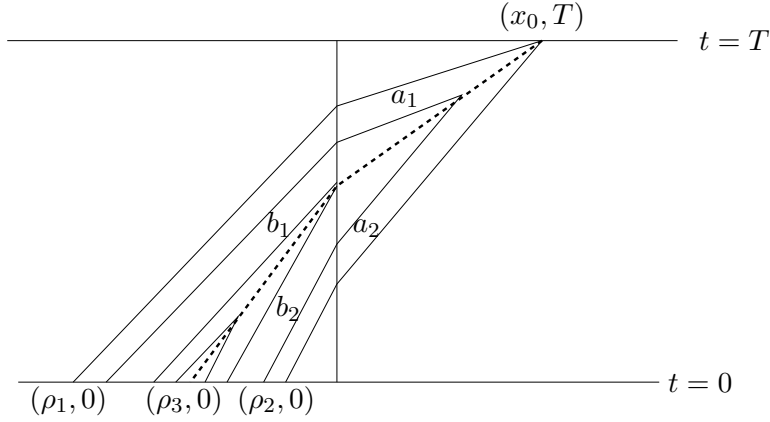}}%
    \put(0.64,0.573){\color[rgb]{0,0,0}\makebox(0,0)[lb]{\smash{$(x_0,T)$}}}%
        \put(0.9,0.539){\color[rgb]{0,0,0}\makebox(0,0)[lb]{\smash{$t=T$}}}%
                \put(0.86,0.09){\color[rgb]{0,0,0}\makebox(0,0)[lb]{\smash{$t=0$}}}%
                                \put(0.03,0.07){\color[rgb]{0,0,0}\makebox(0,0)[lb]{\smash{$(\rho_1,0)$}}}%
                \put(0.3,0.07){\color[rgb]{0,0,0}\makebox(0,0)[lb]{\smash{$(\rho_2,0)$}}}%
                                \put(0.18,0.07){\color[rgb]{0,0,0}\makebox(0,0)[lb]{\smash{$(\rho_3,0)$}}}%
                                                                \put(0.338,0.3){\color[rgb]{0,0,0}\makebox(0,0)[lb]{\smash{$b_1$}}}%
                                \put(0.35,0.19){\color[rgb]{0,0,0}\makebox(0,0)[lb]{\smash{$b_2$}}}%
                                \put(0.5,0.47){\color[rgb]{0,0,0}\makebox(0,0)[lb]{\smash{$a_1$}}}%
                                \put(0.45,0.3){\color[rgb]{0,0,0}\makebox(0,0)[lb]{\smash{$a_2$}}}%

  \end{picture}%
\endgroup
        \caption{The dotted line is the shock originating from the point $(\rho_3,0)$ until the point $(x_0,T).$}
        \label{fig:4}
\end{figure}
Due to the construction of $a_1,a_2,b_1,b_2,s_1,s_2,t_3,\rho_3$,  the solution in the region $x\in \mathbb{R},T>t\geq 0,$ 
to the above initial data is given by (see figure \ref{Fig3})
\begin{eqnarray}
 u(x,t)=
\left\{\begin{array}{lll}
 b_1 &\mbox{ if }& x-\rho_3<s_1t, x<0,\\
b_2 &\mbox{ if}& x-\rho_3>s_1t, x<0,\\
 a_1 &\mbox{ if }& x<s_2(t-t_3),x>0, \\
a_2 &\mbox{ if}& x>s_2(t-t_3),x>0.
\end{array}\right.\label{solutionlemma2}
\end{eqnarray}
This proves the lemma.
 \end{proof}
\begin{lemma}\label{BClemma}
 Let $R>0$. Let us assume that $\rho:[0,R]\rightarrow (-\infty,0)$  and 
 $y:\mathbb{R}\setminus[0,R]\rightarrow \mathbb{R}$ be two non decreasing functions
 satisfies
 \begin{eqnarray}\label{rho-relation}
  \begin{array}{lll}
xy(x)\geq 0   \ \mbox{if}\ x\in \mathbb{R}\setminus[R,0]\ \mbox{and} \ \rho(0)\geq y(x)\ \mbox{if}\ x\leq 0.\\
  \end{array}
 \end{eqnarray}
Then there exists a
 solution $u\in L^\infty(\mathbb{R}\times[0,\infty))$ of \eqref{conlaw-equation}
and  an unique strictly decreasing function $t:[0,R)\rightarrow [0,T)$ such that 
\begin{eqnarray}
  -\frac{\rho(x)}{t(x)}&=&h_+\left(\frac{x}{T-t(x)}\right) \ \mbox{a.e.}\  x\in[0,R), \label{mainrelation}\\
u(x,T)&=&(f^\prime)^{-1}\left(\frac{x}{T-t(x)}\right)\ \mbox{a.e.}\  x\in[0,R). \label{Final-state} \end{eqnarray}
    \end{lemma}
\begin{proof}
In order to prove the Lemma, we split into several steps. In step 1, we construct a solution when $\rho$ is constant. By using step 1,
we allow $\rho$ to be two constants state in step 2.
In step 3, we consider $\rho$ to be an  increasing step function in $[0,R).$ Finally in Step 4, we pass to limit and obtain the result.

\noindent \textbf{Step 1}: {\it Let $0\leq x_1<x_2, T>0$ and $\rho:[x_1,x_2]\rightarrow (-\infty,0)$ be a constant function, 
 then there exists a strictly decreasing function $t: [0,x_2]\rightarrow [0,T]$  and a
 solution $u\in L^\infty(\mathbb{R}\times[0,\infty))$ of \eqref{conlaw-equation}
  satisfies (\ref{mainrelation}), (\ref{Final-state}) for a.e. $x\in[x_1,x_2)$. }
 
\noindent{\it Proof of Step 1.}  Let  $\rho(x)=\rho_0\in (-\infty,0), \ \forall \ x\in [x_1,x_2].$
Let us consider the initial data $u_0$ defined in $\mathbb{R}_-$ by 
 $ u_0(x)= b_1\mathbf{1}_{x<\rho_0}+ b_2\mathbf{1}_{x>\rho_0}$, 
 where $b_1,b_2$ are going to be specify later with the  properties, 
 $   0<g^\prime(b_1)<g^\prime(b_2)$ and 
   $T>-\frac{\rho_0}{g^\prime(b_1)}=t_1\ (\mbox{say})$,   $T>-\frac{\rho_0}{g^\prime(b_2)}=t_2\ (\mbox{say}).
$ 
Then, for $x<0, 0\leq t<T$,  the solution $u(x,t)$ of (\ref{conlaw-equation})  for the above initial data is given by 
\begin{eqnarray}\label{solution1}
  u(x,t)&=&
\left\{\begin{array}{lllll}
b_1 &\mbox{if} \ x<g^\prime(b_1)t+\rho_0,\\
(g^\prime)^{-1}\left(\frac{x-\rho_0}{t}\right) &\mbox{if}\   g^\prime(b_1)t+\rho_0<x<g^\prime(b_2)t+\rho_0,\\
 b_2 &\mbox{if} \  x>g^\prime(b_2)t+\rho_0, 
   \end{array}\right.
  \end{eqnarray}
By R-H condition (\ref{RH-condition}), we define $a_1=f_+^{-1}g(b_1), a_2=f_+^{-1}g(b_2)$ and again by R-H condition and (\ref{solution1}), 
for $t\in[0,T]$, we conclude
\begin{eqnarray}\label{m1}
\begin{array}{ll}
u(0+,t)=a_1\mathbf{1}_{T\geq t>t_1} + f_+^{-1}g(g^\prime)^{-1}\left(\frac{-\rho_0}{t}\right)\mathbf{1}_{t_2<t<t_1}+ a_2\mathbf{1}_{0\leq t< t_2}, 
   \end{array}
  \end{eqnarray}
  hence for $x>0, 0<t\leq T$, the solution is given by 
  \begin{eqnarray}\label{m2}
  u(x,t)&=&
\left\{\begin{array}{lllll}
a_1 &\mbox{if} \ x<f^\prime(a_1)(t-t_1),\\
f_+^{-1}g(g^\prime)^{-1}u(0+,t_+(x)) &\mbox{if}\   f^\prime(a_1)(t-t_1)<x<f^\prime(a_2)(t-t_2),\\
 a_2 &\mbox{if} \ x>f^\prime(a_2)(t-t_2), 
   \end{array}\right.
  \end{eqnarray}  
\begin{figure}[ht]\label{Fig4}
        \centering
        \def\svgwidth{0.9\textwidth}
        \begingroup
    \setlength{\unitlength}{\svgwidth}
  \begin{picture}(1,0.55363752) \put(0,0.1){\includegraphics[width=0.7\textwidth]{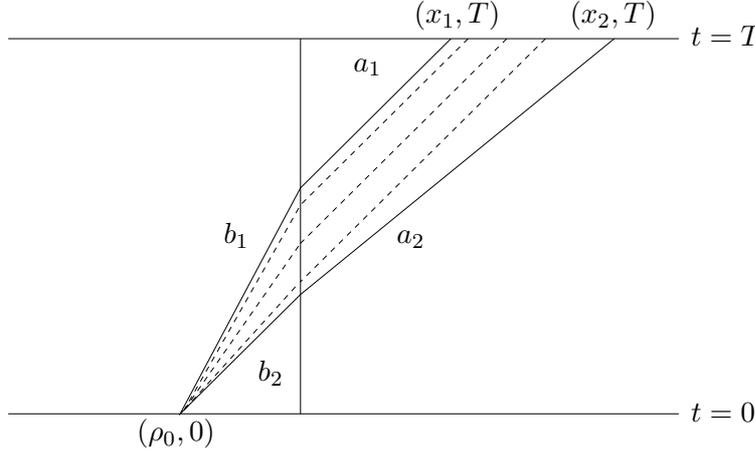}}%
     \put(0.47,0.557){\color[rgb]{0,0,0}\makebox(0,0)[lb]{\smash{$(x_1,T)$}}}%
          \put(0.65,0.557){\color[rgb]{0,0,0}\makebox(0,0)[lb]{\smash{$(x_2,T)$}}}%
        \put(0.79,0.53){\color[rgb]{0,0,0}\makebox(0,0)[lb]{\smash{$t=T$}}}%
                \put(0.79,0.093){\color[rgb]{0,0,0}\makebox(0,0)[lb]{\smash{$t=0$}}}%
                                \put(0.15,0.07){\color[rgb]{0,0,0}\makebox(0,0)[lb]{\smash{$(\rho_0,0)$}}}%
                                                                \put(0.25,0.3){\color[rgb]{0,0,0}\makebox(0,0)[lb]{\smash{$b_1$}}}%
                                \put(0.29,0.14){\color[rgb]{0,0,0}\makebox(0,0)[lb]{\smash{$b_2$}}}%
                                \put(0.4,0.5){\color[rgb]{0,0,0}\makebox(0,0)[lb]{\smash{$a_1$}}}%
                                \put(0.45,0.3){\color[rgb]{0,0,0}\makebox(0,0)[lb]{\smash{$a_2$}}}%
  \end{picture}%
\endgroup
        \caption{A rarefaction originating from the point $(\rho_0,0)$.}
        \label{fig:5}
\end{figure}
  where $t_+:[f^\prime(a_1)(T-t_1),f^\prime(a_2)(T-t_2)]\rightarrow [t_1,t_2]$ is a homeomorphism, the existence of 
  such homeomorphism is quite obvious.  
   Note that the lines $x=f^\prime(a_1)(t-t_1), x=f^\prime(a_2)(t-t_2)$, hits $t=T$ at $x=f^\prime(a_1)(T-t_1)$ and  $x=f^\prime(a_1)(T-t_2)$ respectively. 
 Now we are interested to solve the following equation for $(a_1,a_2)$ and $(b_1,b_2)$, i.e, given $x_1,x_2,\rho_0$,
  find pairs $(a_1,a_2)$
 and $(b_1,b_2)$ such that
 \begin{eqnarray}\label{BRiemann}
  \begin{array}{lllll}
   x_1=f^\prime(a_1)(T+\frac{\rho_0}{g^\prime(b_1)}),\ x_2=f^\prime(a_2)(T+\frac{\rho_0}{g^\prime(b_2)}).
  \end{array}
 \end{eqnarray}
In order to solve the above equation (\ref{BRiemann}), we need to consider following 2 cases. 

\noindent \textbf{Case 1:} If $f(\theta_f)\leq g(\theta_g)$: We consider the function $S_1:[\theta_g,\infty]\rightarrow \mathbb{R}$ defined by 
$S_1(x)=f^\prime g^{-1}_+f(x)\left(T+\frac{\rho_0}{g^\prime(x)}\right)$. Then $S_1$ is a continuous function with
$S_1(\theta_g)=-\infty$, $S_1(\infty)=\infty$. Therefore 
by using Intermediate Value Theorem we conclude  the existence of pairs $(a_1,a_2),(b_1,b_2)$ satisfying  \eqref{BRiemann}.

\noindent \textbf{Case 2:} If $f(\theta_f)> g(\theta_g)$: The argument is similar to case 1. 

Now we define  a strictly decreasing function $t:[0,x_2]\rightarrow [0,T]$, by 
\begin{eqnarray}\label{m3}
 t(x)= \left\{\begin{array}{lll}
          T-\frac{x}{f^\prime(a_1)}  &\mbox{if}& 0\leq x<f^\prime(a_1)(t-t_1),\\
          t_+(x) &\mbox{if}&   f^\prime(a_1)(t-t_1)<x<f^\prime(a_2)(t-t_2),\\
                    T-\frac{x}{f^\prime(a_2)}  &\mbox{if}& x_2> x>f^\prime(a_2)(t-t_2).  \end{array}\right.
\end{eqnarray}
From (\ref{solution1}), (\ref{m1}), (\ref{m2}) and (\ref{m3}),  it is easy to check (\ref{mainrelation}), (\ref{Final-state}) for 
a.e. $x\in [x_1,x_2].$
Which proves Step 1.

\noindent \textbf{Step 2}: {\it Let $0\leq x_1<x_2<x_3, T>0$ and $\rho:[x_1,x_3]\rightarrow (-\infty,0)$ be such that 
$\rho(x)= \rho_1 \mathbf{1}_{[x_1,x_2]} + \rho_2\mathbf{1}_{[x_2,x_3]} 
$
where $\rho_1,\rho_2$ are two constants such that $\rho_1<\rho_2<0.$
 Then there exists a strictly decreasing function $t: [0,x_3]\rightarrow [0,T]$ 
 and a
 solution $u\in L^\infty(\mathbb{R}\times[0,\infty))$ of \eqref{conlaw-equation}
  satisfies (\ref{mainrelation}), (\ref{Final-state}) for $\mbox{a.e.}\  x\in[0,x_3)$.}
 
 \noindent{\it Proof of Step 2.} Consider the function $\rho$ in $[x_1,x_2]$, then by Step 1, there exists  pairs $(a_1,a_2)$ 
(say), $(t_1,t_2)$ (say) and $(b_1,b_2)$ (say)
as in (\ref{BRiemann}). Similarly considering  the function $\rho$ in $[x_2,x_3]$ and using Step 1, there exists other
pairs $(a_3,a_4)$, $(t_3,t_4)$ (say) and $(b_3,b_4)$  as in (\ref{BRiemann}). 
Then by construction $t_2>t_3$ and it satisfies 
\begin{eqnarray*}
  -\frac{\rho_1}{t_2}=h_+\left(\frac{x_2}{T-t_2}\right),\
   -\frac{\rho_2}{t_3}=h_+\left(\frac{x_2}{T-t_3}\right).
 \end{eqnarray*}
  \begin{figure}[ht]\label{Fig5}
        \centering
        \def\svgwidth{0.9\textwidth}
        \begingroup
    \setlength{\unitlength}{\svgwidth}
  \begin{picture}(1,0.55363752)%
    \put(0,0.1){\includegraphics[width=0.7\textwidth]{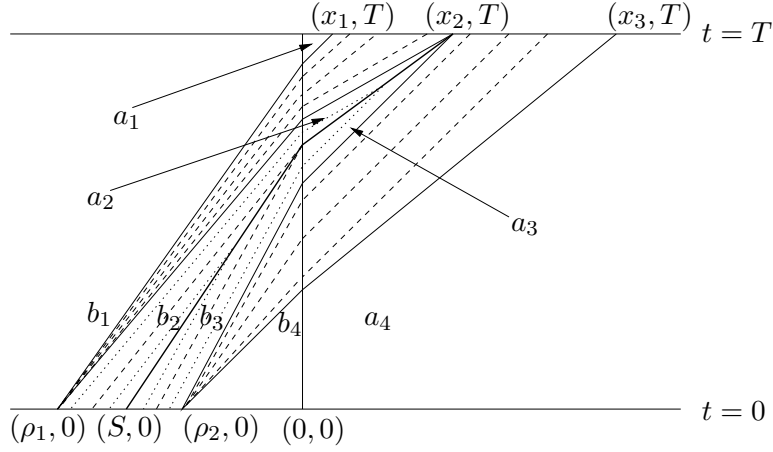}}%
    \put(0.8,0.09){\color[rgb]{0,0,0}\makebox(0,0)[lb]{\smash{$t=0$}}}%
        \put(0.314,0.069){\color[rgb]{0,0,0}\makebox(0,0)[lb]{\smash{$(0,0)$}}}%
    \put(0.8,0.53){\color[rgb]{0,0,0}\makebox(0,0)[lb]{\smash{$t=T$}}}%
                \put(0.347,0.5514){\color[rgb]{0,0,0}\makebox(0,0)[lb]{\smash{$(x_1,T)$}}}%
                                \put(0.0,0.0714){\color[rgb]{0,0,0}\makebox(0,0)[lb]{\smash{$(\rho_1,0)$}}}%
                \put(0.2,0.0714){\color[rgb]{0,0,0}\makebox(0,0)[lb]{\smash{$(\rho_2,0)$}}}%
                                \put(0.1,0.0714){\color[rgb]{0,0,0}\makebox(0,0)[lb]{\smash{$(S,0)$}}}%
            \put(0.48,0.5514){\color[rgb]{0,0,0}\makebox(0,0)[lb]{\smash{$(x_2,T)$}}}%
                            \put(0.09,0.204){\color[rgb]{0,0,0}\makebox(0,0)[lb]{\smash{$b_1$}}}%
                                                        \put(0.17,0.2){\color[rgb]{0,0,0}\makebox(0,0)[lb]{\smash{$b_2$}}}%
                            \put(0.22,0.2){\color[rgb]{0,0,0}\makebox(0,0)[lb]{\smash{$b_3$}}}%
                            \put(0.31,0.192){\color[rgb]{0,0,0}\makebox(0,0)[lb]{\smash{$b_4$}}}%
                            \put(0.12,0.434){\color[rgb]{0,0,0}\makebox(0,0)[lb]{\smash{$a_1$}}}%
                            \put(0.09,0.34){\color[rgb]{0,0,0}\makebox(0,0)[lb]{\smash{$a_2$}}}%
                            \put(0.58,0.31){\color[rgb]{0,0,0}\makebox(0,0)[lb]{\smash{$a_3$}}}%
                            \put(0.41,0.2){\color[rgb]{0,0,0}\makebox(0,0)[lb]{\smash{$a_4$}}}%
                        \put(0.69,0.5514){\color[rgb]{0,0,0}\makebox(0,0)[lb]{\smash{$(x_3,T)$}}}%
  \end{picture}%
\endgroup
        \caption{An illustration Step 2.}
        \label{fig:6}
\end{figure}
 Now by Lemma \ref{discontinuouslemma} and  Step 1, there exists $S\in (\rho_1,\rho_2)$ which allow us to  construct the following initial data 
defined  in $\mathbb{R}_-$ by 
   $u_0(x)=b_1\mathbf{1}_{x<\rho_1}+b_2\mathbf{1}_{\rho_1<x<S}$ $+b_3\mathbf{1}_{S<x<\rho_2}$ $+b_4\mathbf{1}_{x>\rho_2}.$
Then the corresponding solution in the region $\{x<0,0\leq t\leq T\}$ is given by (see figure \ref{Fig5})
\begin{eqnarray}\label{solution-for-two-constant,x>0}
  u(x,t)=
\left\{\begin{array}{lllll}
b_1 &\mbox{if}&  x<g^\prime(b_1)t+\rho_1,\\
(g^\prime)^{-1}\left(\frac{x-\rho_1}{t}\right) &\mbox{if}& g^\prime(b_1)t+\rho_1<x<g^\prime(b_2)t+\rho_1,  \\
 b_2 &\mbox{if}&  g^\prime(b_2)t+\rho_1<x<\frac{g(b_2)-g(b_3)}{b_2-b_3}t+S,\\
  b_3 &\mbox{if}&  \frac{g(b_2)-g(b_3)}{b_2-b_3}t+S<x<g^\prime(b_3)t+\rho_2,\\
(g^\prime)^{-1}\left(\frac{x-\rho_2}{t}\right) &\mbox{if}& g^\prime(b_3)t+\rho_2<x<g^\prime(b_4)t+\rho_2,  \\
  b_4 &\mbox{if}&  g^\prime(b_4)t+\rho_2<x.\\
\end{array}\right.
  \end{eqnarray} 
  By R-H condition (\ref{RH-condition}), we define $a_i=f_+^{-1}g(b_i),$ for $i=1,\cdots, 4.$ Again by R-H condition and
  (\ref{solution-for-two-constant,x>0}), 
  the solution in the region  $\{x>0, 0<t\leq T\}$,  is given by 
  \begin{eqnarray}\label{m4}
  u(x,t)=
\left\{\begin{array}{lllll}
a_1  &\mbox{if}&  x<f^\prime(a_1)(t-t_1),\\
f_+^{-1}g(g^\prime)^{-1}u(0+,t_+(x)) &\mbox{if}&  f^\prime(a_1)(t-t_1) <x<f^\prime(a_2)(t-t_2),\\
  a_2  &\mbox{if}& f^\prime(a_2)(t-t_2)<x <\tilde{S}_1 \left(t+\frac{S}{\tilde{S}_2}\right),\\
 a_3 &\mbox{if}&
\tilde{S}_1 \left(t+\frac{S}{\tilde{S}_2}\right)<x<f^\prime(a_3)(t-t_3),\\
 f_+^{-1}g(g^\prime)^{-1}u(0+,t_+(x))&\mbox{if}&   f^\prime(a_3)(t-t_3) <x<f^\prime(a_4)(t-t_4),\\
 a_4  &\mbox{if}& f^\prime(a_4)(t-t_4)>x, 
   \end{array}\right.
  \end{eqnarray}
  where $\tilde{S_1}=\left(\frac{f(a_2)-f(a_3)}{a_2-a_3}\right)$, $\tilde{S_2}=\frac{g(b_2)-g(b_3)}{b_2-b_3}$ and $t_+:[f^\prime(a_1)(T-t_1),f^\prime(a_2)(T-t_2)]\cup f^\prime(a_3)(T-t_3),f^\prime(a_4)(T-t_4)] \rightarrow [t_1,t_2]\cup [t_3,t_4]$ is 
  a homeomorphism.
 Then  we define a strictly decreasing function $t: [0,x_3]\rightarrow [0,T]$ by 
$
 t(x)=   \left(T-\frac{x}{f^\prime(a_1)}\right) \mathbf{1}_{[0,f^\prime(a_1)(t-t_1)]}+ 
t_+(x)\mathbf{1}_{[0,x_3]\setminus[0,f^\prime(a_1)(t-t_1)]}.$
                 Therefore from definition of $t$, (\ref{solution-for-two-constant,x>0}) and  (\ref{m4}),  it is easy 
         to check (\ref{mainrelation}), (\ref{Final-state}) for 
a.e. $x\in [0,x_3].$
Which proves Step 2.
 
\noindent\textbf{Step 3:} {\it Discretization of  both the functions $\rho,y$ by piecewise constant  and develop a solution  with a 
piecewise constant initial data such that 
(\ref{mainrelation}), (\ref{Final-state}) holds for each discretized function $\rho_N.$}
 
 In the present step our aim to create a piecewise constant initial data in the region $[y(0),y(R)]$. Initial data $\bar{u}^N_0$ (say)
 in the region $\mathbb{R}\setminus 
 [y(0),y(R)]$ can be construct in the same way as in Lemma 3.6 of \cite{Sop}.
 In order to do that we first discretized $\rho$ to piecewise constants. Let $N\in \mathbb{N}$. 
 Let $\rho(0)=z_1<z_2<\cdots<z_N=\rho(R)$ be such that 
 $  |z_i-z_{i+1}|<\frac{1}{N}$ for $i=1,\cdots, N-1$.
We define $\rho^{-1}[z_1,z_{i+1}]=[x_0,x_i],$ then $0=x_0\leq x_1\cdots \leq x_N=R. $
Let us define a new function $\rho_N:[0,R]\rightarrow (-\infty,0)$ 
by $ \rho_N(x)=z_1\mathbf{1}_{[x_0,x_1]}+\sum\limits_{i=2}^{N-1}z_i\mathbf{1}_{(x_i,x_{i+1}]}(x).$
By definition of $\rho_N$, we have 
 $|\rho_N-\rho(x)|<\frac{1}{N}, \ \mbox{for} \ x\in(0,R).
$   
 For $i=1,\cdots, N-1,$ we consider $\rho_N$ in each interval $[x_i,x_{i+1}]$ and apply step 1 and step 2, then for each $[x_i,x_{i+1}]$ 
 there exist
pairs $(b_{2i+1}, b_{2i+2}), (a_{2i+1}, a_{2i+2}), (t_{2i+1}, t_{2i+2})$ and $S_i\in (z_i,z_{i+1})$. Also the following equation satisfies
for $i=1,\cdots, N-1,$
\begin{eqnarray*}
  -\frac{\rho_i}{t_{2i}}&=h_+\left(\frac{x_i}{T-t_{2i}}\right),\
   -\frac{\rho_{i+1}}{t_{2i+1}}&=h_+\left(\frac{x_i}{T-t_{2i+1}}\right).
\end{eqnarray*}
\begin{figure}[ht]\label{Fig6}
        \centering
        \def\svgwidth{0.8\textwidth}
        \begingroup
    \setlength{\unitlength}{\svgwidth}
  \begin{picture}(1,0.55363752)%
    \put(0,0.1){\includegraphics[width=0.9\textwidth]{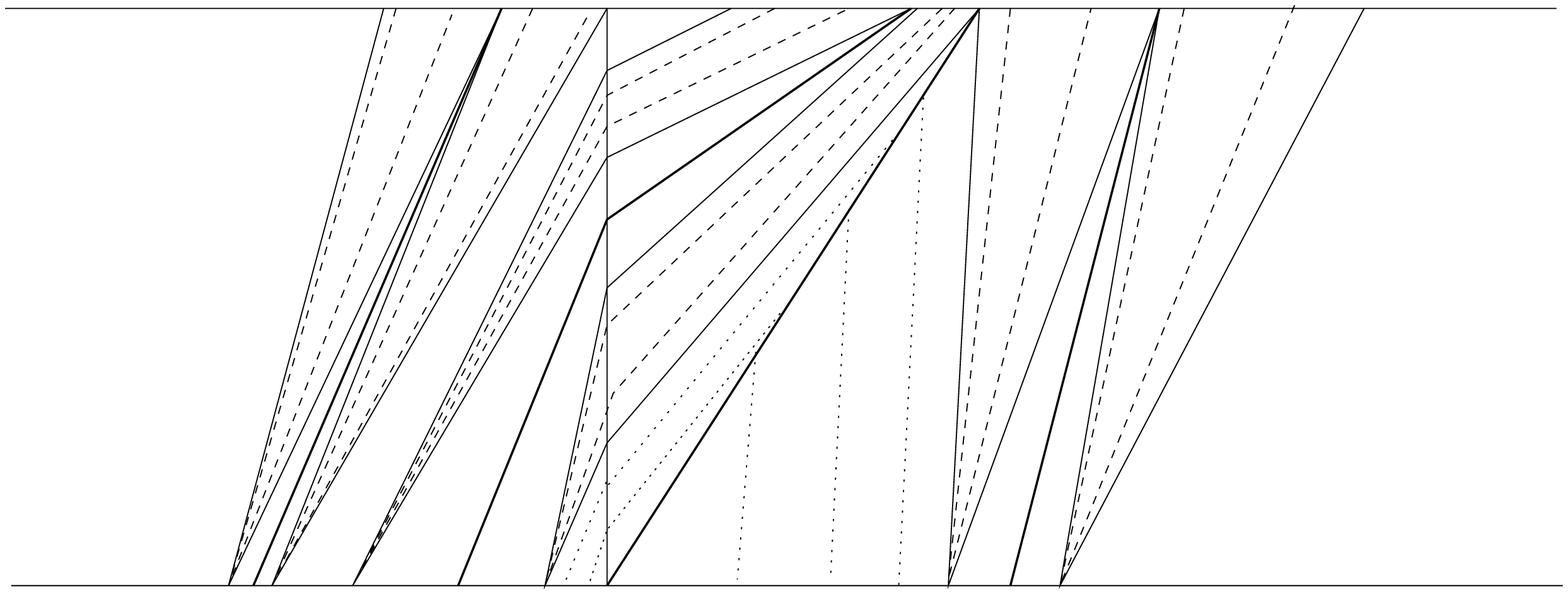}}%
    \put(-0.085,0.09){\color[rgb]{0,0,0}\makebox(0,0)[lb]{\smash{$t=0$}}}%
        \put(0.4,0.077){\color[rgb]{0,0,0}\makebox(0,0)[lb]{\smash{$x=0$}}}%
    \put(-0.085,0.5109){\color[rgb]{0,0,0}\makebox(0,0)[lb]{\smash{$t=T$}}}%
        \put(0.2185,0.07){\color[rgb]{0,0,0}\makebox(0,0)[lb]{\smash{$(\rho_i,0)$}}}%
                \put(0.632185,0.07){\color[rgb]{0,0,0}\makebox(0,0)[lb]{\smash{$(y_k,0)$}}}%
                                \put(0.693,0.53){\color[rgb]{0,0,0}\makebox(0,0)[lb]{\smash{$(R,T)$}}}%
            \put(0.47,0.53){\color[rgb]{0,0,0}\makebox(0,0)[lb]{\smash{$(x_i,T)$}}}%
                        \put(0.57,0.53){\color[rgb]{0,0,0}\makebox(0,0)[lb]{\smash{$(x_{i+1},T)$}}}%
  \end{picture}%
\endgroup
        \caption{An illustration of Step 3.}
        \label{fig:7}
\end{figure}
Hence we obtain the following piecewise constant initial data in the region $[z_1,z_N]$, combining $\bar{u}_0\in \mathbb{R}\setminus [y(0),y(R)]$
we obtain the following initial data $u_0^N$ (say), given by (see figure \ref{Fig6})
 \begin{eqnarray}\label{backfinal}
  u_0^N(x)=
\left\{\begin{array}{lllll}
\bar{u}^N_0(y(0)-) &\mbox{if}&  x\in[y(0),\bar{S}^N],\\
b_1 &\mbox{if}&  x\in[S_0,z_1],\\
b_{2i} &\mbox{if}&  x\in[z_i,S_i],\ \mbox{for}\ i=1,\cdots, N-1,\\
 b_{2i+1} &\mbox{if}&  x\in[S_i,z_{i+1}],\ \mbox{for}\ i=1,\cdots, N-1,\\
 b_{2N} &\mbox{if}&  x\in[z_N,0],\ \mbox{for}\ i=1,\cdots, N-1,\\
 \bar{u}^N_0(y(R)+) &\mbox{if}&  x\in[0,y(R)],\\
 \bar{u}^N_0 &\mbox{if}&  x\in\mathbb{R}\setminus [y(0),y(R)],\\
   \end{array}\right.
  \end{eqnarray}
where $-\bar{S}^N=\displaystyle\frac{g(\bar{u}^N_0(y(0)-))-g(b_1)}{\bar{u}^N_0(y(0)-)-b_1}.$
Now similarly as in Step 1, Step 2 there is a homeomorphism denoting  by 
$t^N:\cup_{i=1}^{N-1}[f^\prime(a_i)(T-t_i),f^\prime(a_2)(T-t_2)] \rightarrow \cup_{i=1}^{N-1} [t_{i},t_{i+1}]$.
 Which is a strictly decreasing function  from $[0,R]\rightarrow [0,T]$ and by using the explicit formula  
 it is easy 
         to check (\ref{mainrelation}), (\ref{Final-state}) for 
a.e. $x\in [0,R].$ Let us denote the corresponding solution $u^N$ of the initial data  $u_0^N$, then it is to be noticed that the construction
has been done such a way that the shocks are discrete in the region $[0,T]\times \mathbb{R}.$

\noindent\textbf{Step 4:} {\it Passage  to the limit: in this step we prove up to a subsequence $u_0^N$ converges to some $u_0$ in $L^1_{loc}$ and the corresponding solution $u^N$ 
also converges to the solution $u$ up to a subsequence  and finally (\ref{mainrelation}), (\ref{Final-state}) holds.}

By definition, $\rho^N\rightarrow \rho(x)$ point-wise and by Helly's theorem there exists a subsequence (after relabeling) 
such that $t^N(x)\rightarrow t(x)$ (say) a.e.. Hence the relation (\ref{mainrelation}) holds for $\rho(\cdot), t(\cdot)$, for a.e. $x\in [0,R].$
 Let us fix $C_1\in (0,R)$. Since from Lemma \ref{simplelemma}, $t(\cdot)$ is strictly decreasing in $[0,R],$ hence there exists a constant 
 $C_2>0,$ such that for  $x\in (0,C_1)$, $t_+(x)>C_2>0$ and $\frac{x}{T-t(x)}<\frac{R}{T-t(C_1)}$ for $x\in(C_1,R)$.  Therefore, there exists 
 constant $C_3>0$ such that 
 \begin{eqnarray}
  \left|\frac{\rho(x)}{t(x)}\right|&\leq \frac{\rho(0)}{C_2} \quad \mbox{if} \quad x\in(0,R),\label{r1}\\
 \left|h_+\left(\frac{x}{T-t(x)}\right)\right|&\leq C_3 \quad \mbox{if} \quad x\in(C_1,R)\label{r2}.
                                                                                                      \end{eqnarray}
Whence (\ref{r1}), (\ref{r2}) and Step 3  allow us to assume that up to a subsequence (after relabeling) 
\begin{eqnarray}\label{r3}
\begin{array}{ll}
 \left|\frac{\rho(x)}{t^N(x)}\right| =\left|h_+\left(\frac{x}{T-t^N(x)}\right)\right|\leq C_4 \quad \mbox{if}\quad x\in[0,R],
 \end{array}
\end{eqnarray}
for some constant $C_4>0.$ By using the explicit formula and (\ref{r3}), there exist a constant $C_5$,
such that 
\begin{eqnarray}\label{r4}
\begin{array}{ll}
 \mbox{Max}\left\{||u_0^N||_\infty, ||u_N||_\infty, \left\| \frac{d\xi}{dt} \right\|_\infty\right\} \leq  C_5,
\end{array}
\end{eqnarray}
where $\xi$ be any characteristic associated to initial data $u_0^N.$ 
Let us consider any partition $\{p_i\}_{i=1}^K$ for the interval $(\rho(0),\rho(R)).$ Then by  explicit formula and by our construction in Step 3,
there exists corresponding partition  $\{t_i\}_{i=1}^K$ in the interval $(0,T)$ such that 
$g^\prime(u_0^N(p_i)) =-\frac{y^N_{-,0}(t_{i+1})}{t_{i+1}},$ moreover due to the fact that $\rho(R)<0$ and (\ref{r4}), 
\begin{equation}\label{r6}
 \frac{1}{t_i}\leq C_6,
\end{equation}
for some constant $C_6>0.$
Now by using explicit formula and using (\ref{r6}) there exist a constant $C_7>0$ such that 
\begin{eqnarray}\label{BVestimate}
\begin{array}{lllll}
TV(g^\prime(u_0^N): (\rho(0),\rho(R)))&=&\sup\limits_{K}\sum\limits_{i=1}^{K} \left|\frac{y^N_{-,0}(t_{i+1})}{t_{i+1}}-\frac{y^N_{-,0}(t_{i})}{t_{i}}\right|\\
&\leq&\sup\limits_{K}(C_6)^2\sum\limits_{i=1}^{K}|t_{i+1}-t_{i}||y^N_{-,0}(t_{i+1})|\\
&+& \sup\limits_{K}(C_6)^2\sum\limits_{i=1}^{K}|t_i||y^N_{-,0}(t_{i+1})-y^N_{-,0}(t_{i})|\\
&\leq &\sup\limits_{K}(C_6)^2\{|\rho(0)|T+T|\rho(0)-\rho(R)|\}\\
&= &T(C_6)^2\{|\rho(0)|+|\rho(0)-\rho(R)|\}.
\end{array} 
\end{eqnarray} Similarly as in (\ref{BVestimate}) 
one can prove $BV_{loc}(g^\prime(u_0^N(0-,t)): (0,T))\leq C_7$, for all $N\in \mathbb{N}$.
Thanks to Helly's Theorem, there exists subsequence   (after relabeling)  such that $\{g^\prime(u_0^N)\}$ converges point-wise to some function 
$Q$ (say) in $(\rho(0),\rho(R))$ . Then define $\tilde{u}_0(x)=(g^\prime)^{-1}Q(x),$ therefore $u_0^N\rightarrow \tilde{u}_0$ in
 $L^1(\rho(0),\rho(R))$. It can be shown that $\bar{u}^N_0 \rightarrow \bar{u}_0$ (say) in $L^1(\mathbb{R}\setminus 
 [y(0),y(R)])$.
  Again by Helly's Theorem, there exists subsequence   (after relabeling)  $u_0^N(0-,t)\rightarrow u(0-,t)$ (say) a.e. $t\in[0,T]$ and so 
$u_0^N(0+,t)\rightarrow u(0+,t)$ (say) for a.e. $t\in [0,T]$. Now we consider the following two boundary value problems
\begin{equation}\label{conlaw-equationB1}
 \left\{ \begin{split}
W^N_t+f(W^N)_x&=0, & \mbox{if}& \  x>0, t\in [0,T]\\
W^N(t,0)&=u^N(0+,t), & \mbox{if}& \ t\in [0,T]\\
  W^N(x,0)&=u^N_0(x)\mid_{(0,\infty)}, & \mbox{if}&\  x>0.
 \end{split}
 \right.
\end{equation}
\begin{equation}\label{conlaw-equationB2}
 \left\{ \begin{split}
V^N_t+g(V^N)_x&=0,& \mbox{if}& \ x<0, t\in [0,T]\\
V^N(t,0)&=u^N(0-,t),& \mbox{if}& \ t\in [0,T]\\
  V^N(x,0)&=u^N_0(x)\mid_{(-\infty,0)},& \mbox{if}& \  x<0.
 \end{split}
 \right.
\end{equation}
Then one can follow a similar strategy as in 'proof of Lemma 2.1 and 2.2' of \cite{Sco} to conclude that $W^N\rightarrow W$,  $V^N\rightarrow V$
in $L^1_{loc}$ and the limits $W,V$ is the entropy solutions to  the above boundary value problems with boundary data $u(0+,t),u(0-,t)$ and the initial data
$u_0\mid_{(0,\infty)}, u_0\mid_{(-\infty,0)}$ respectively. Then define a new function $Z^N: \mathbb{R}\times [0,T]\rightarrow \mathbb{R}$
by $Z^N=V^N\mathbf{1}_{\mathbb{R}_-\times [0,T]}+W^N\mathbf{1}_{\mathbb{R}_+\times [0,T]}.$ Therefore $Z^N\rightarrow Z$ in $L^1_{loc},$
for some $Z$. It is easy to check that $Z^N$ is a weak solution 
of (\ref{conlaw-equation}). 
Due to the construction, $Z^N$   satisfies R-H condition, interface entropy condition and (\ref{mainrelation})  for each $N$ 
and hence these properties holds in the limit $Z$.
Again  $W^N(x,T)=(f^\prime)^{-1}\left(\frac{x}{T-t^N(x)}\right)$ for  a.e. $x\in (0,R)$ and passing to the limit up to a subsequence
$Z(x,T):=W(x,T)=
(f^\prime)^{-1}\left(\frac{x}{T-t(x)}\right)$  a.e. $x\in (0,R)$. 
This completes the proof of the Lemma.
 \end{proof}
 The following Lemma holds in the same spirit as Lemma \ref{BClemma}.
\begin{lemma}\label{BClemma2}
 Let $R<0$. Let us assume that $\rho:[R, 0]\rightarrow (0,\infty)$  and 
 $y:\mathbb{R}\setminus[R,0]\rightarrow \mathbb{R}$ be two non decreasing functions satisfies
 \begin{eqnarray}\label{rho-relation2}
  \begin{array}{lll}
xy(x)\geq 0  \ \mbox{if}\ x\in \mathbb{R}\setminus[R,0]\ \mbox{and}\
y(x)\geq \rho(0)   \ \mbox{if}\ x\geq 0.\\
  \end{array}
 \end{eqnarray}

 Then there exists a
 solution $u\in L^\infty(\mathbb{R}\times[0,\infty))$ of \eqref{conlaw-equation}
and  an unique strictly increasing function $t:[R,0]\rightarrow [0,T]$ such that for a.e.  $x\in[R,0],$ we have 
\begin{eqnarray*}
  -\frac{\rho(x)}{t(x)}=h_-\left(\frac{x}{T-t(x)}\right),\ 
u(x,T)=(g^\prime)^{-1}\left(\frac{x}{T-t(x)}\right). \end{eqnarray*}
    \end{lemma}
\section{Optimal control for discontinuous flux}\label{section:optimalcontrol}
\noindent \textbf{Class of target function}: Let $0<C$. Let $k$ be 
 a measurable function on $\mathbb{R}.$ Define a new function $\eta: \mathbb{R}\rightarrow \mathbb{R}$ by 
  $\eta[k](x)= g^\prime(k(x))\mathbf{1}_{x\leq 0}+ f^\prime(k(x))\mathbf{1}_{x> 0}.$
Then Supp $\eta[k](x)\subset [-C,C]$ if and only if $k(x)=\theta_g$ for  $x<-C$ and $k(x)=\theta_f$ for  $x>C.$
 
 \noindent\textbf{Admissible class of initial data}:  Let us define a new function $\bar{\theta}:\mathbb{R}\rightarrow\mathbb{R}$ by 
 $\bar{\theta}(x)=H(x)\theta_f+(1-H(x))\theta_g$. Then  Admissible class of initial data  $A$, is defined by 
 \begin{eqnarray}\label{admissible}
\begin{array}{llll}
A =\{u_0(x)=z(x)+ \bar{\theta}(x)\ :\ z\in L^\infty(\mathbb{R})\ \mbox{with compact support}
\}.\end{array}\end{eqnarray}
\noindent\textbf{Cost functional}:
 Fix a $T>0$ and for $u_0\in A$, let $u$ be the corresponding solution to (\ref{conlaw-equation}) associated to the initial data $u_0$.
 We define the cost functional $J$ by 
 \begin{eqnarray}
 \begin{array}{lllll}\label{cost1}
J(u_0)&= \int\limits_{-\infty}^{L_1(T)}|g^\prime(u(x,T)) -
\eta[k](x)|^2dx +
\int\limits_{L_1(T)}^{0}| f^\prime f_-^{-1}g(u(x,T))-\eta[k](x)|^2dx\\
&+ \int\limits_{0}^{R_1(T)}|g^\prime g_+^{-1}f(u(x,T))-\eta[k](x)|^2dx+\int\limits_{R_1(T)}^{\infty}|f^\prime(u(x,T)) -
\eta[k](x)|^2dx.
 \end{array}
\end{eqnarray}
\noindent \textbf {Optimal control problem:}
Then the question is to find the  optimal control $u_0\in A$ so that, one has
\begin{eqnarray}
 J(u_0)=\min_{w_o\in A}{J}(w_0).\label{optimalcontrol}
\end{eqnarray}

 \begin{theorem}\label{maintheoremop}
  There exists a minimizer for (\ref{optimalcontrol}).
  \end{theorem}
  
Let us define the following admissible class of initial data. Let $T>0.$ In order to mention simple notations  we denote $R(t),L(t)$ instead of 
$R_1(t), L_1(t).$

\par We consider the following two admissible class of initial data
\begin{eqnarray*}
A_1 =\{u_0\in A : \ L(T,u_0)=0\}\ \mbox{and}\ A_2 =\{u_0\in A : \ R(T,u_0)=0
\}.\end{eqnarray*}
From (6) of Theorem \ref{AG1}, $A=A_1\cup A_2.$ $ R(T,u_0),  L(T,u_0),$ denotes the $R_1,L_1$ curves as in Theorem \ref{AG1} with respect 
to the initial data $u_0$.

In view of (6) of Theorem \ref{AG1}, we conclude 
\begin{eqnarray}
 \min_{u_o\in A}{J}(u_0)=\min\{ \min_{u_o\in A_1}{J}(u_0),  \min_{u_o\in A_2}{J}(u_0)\}.\label{optimalcontrol2}
\end{eqnarray}
Hence finding  a minimum in \eqref{optimalcontrol} is equivalent to find minimum of the functional $J$ over the sets $A_1,A_2$. 

\begin{lemma}
 For $u_0\in A$, $J(u_0)$ is well defined.
 \begin{proof}
  Because of finite speed of propagation, it is immediate.
 \end{proof}

\end{lemma}

\noindent\underline{\textbf{Existence of minimizer over the set $A_1$}}:
Let us define a new admissible set
\begin{eqnarray*}
\begin{array}{llll}
 \tilde{A}_1=\big\{(R(T),\rho,y): &i).\  \rho:[0,R(T)]\rightarrow (-\infty,0] \ \mbox{be a non increasing function}, \notag\\         
 &ii).\  y:\mathbb{R}\setminus [0,R(T)]\rightarrow\mathbb{R} \ \mbox{be a non decreasing function},\notag\\   
 &iii).\  xy(x)\geq 0, \ \mbox{and}\ y(x)\leq \rho(0)\ \mbox{for all}\ x\leq0\big\}.
 \end{array}
\end{eqnarray*}

From Lemma \ref{BClemma}, for $(R(T),\rho,y)$, there exists a unique non increasing function $t$ such that 
$$-\frac{\rho(x)}{t(x)}=h_+\left(\frac{x}{T-t(x)}\right),\ \mbox{a.e.} \ x\in(0,R(T)).$$
Let us define a new cost functional $\tilde{J}$ associated with the admissible set $\tilde{A}_1$ by 
 \begin{eqnarray}
 \begin{array}{lllll}\label{newcost}
\tilde{J}(R(T),\rho,y)= \int\limits_{-\infty}^{0}|\frac{x-y(x)}{T} -
\eta[k](x)|^2dx &+& \int\limits_{0}^{R(T)}|-\frac{\rho(x)}{t(x)} -
\eta[k](x)|^2dx\\&+&\int\limits_{R(T)}^{\infty}|\frac{x-y(x)}{T} -
\eta[k](x)|^2dx.
 \end{array}
\end{eqnarray}
Then from (\ref{cost1}), we have 
\begin{equation}\label{ine1}
 \inf_{(R(T),\rho,y) \in \tilde{A}_1} \tilde{J}(R(T),\rho,y) \leq \inf_{u_0\in A_1} J(u_0).
\end{equation}

\noindent\underline{\textbf{Estimations}}: Let Supp $\eta[k]\subset [-C,C].$ $(0,0,x)\in \tilde{A}_1$ and $\tilde{J}(0,0,x)=\|\eta[k]\|^2_{L^2}.$
Hence
\begin{eqnarray*}
 \inf_{(R(T),\rho,y) \in \tilde{A}_1} \tilde{J}(R(T),\rho,y) \leq \|\eta[k]\|^2_{L^2}.
\end{eqnarray*}
Let $(R(T),\rho,y) \in \tilde{A}_1$ be such that $\tilde{J}(R(T),\rho,y) \leq 2\|\eta[k]\|^2_{L^2}.$
Suppose $R(T)>C,$ then 
\begin{eqnarray*}
\begin{array}{ll}
 2 \|\eta[k]\|^2_{L^2} \geq  \tilde{J}(R(T),\rho,y)\geq \int\limits_C^{R(t)}\left|\frac{\rho(x)}{t(x)}\right|^2dx
  =\int\limits_C^{R(t)}\left|h_+\left(\frac{x}{T-t(x)}\right)\right|^2dx.
\end{array}
\end{eqnarray*}
Since $h_+$ is an increasing function and $\left(\frac{x}{T-t(x)}\right)\geq \frac{x}{T},$ we obtain 
$ 2\|\eta[k]\|^2_{L^2}\geq$\\$ \int\limits_C^{R(t)}\left|h_+\left(\frac{x}{T}\right)\right|^2dx\rightarrow\infty \ \mbox{as}\ R(T)\rightarrow\infty.$
Hence there exists $R_0\geq C$ such that 
 $R(T)\leq R_0.$
Now for $x\leq 0, y(x)\leq \rho(0)$, which implies $x-y(x)\geq x-\rho(0)>0.$ Hence 
\begin{eqnarray*}
\begin{array}{llll}
  2 \|\eta[k]\|^2_{L^2} \geq \int\limits_{\rho(0)}^{0}\left|\frac{x-y(x)}{T}-\eta[k](x)\right|^2dx
    &\geq& \frac{1}{2}\int\limits_{\rho(0)}^{0}\left|\frac{x-y(x)}{T}\right|^2dx- \|\eta[k]\|^2_{L^2}\\
        &\geq& \frac{1}{2}\int\limits_{\rho(0)}^{0}\left|\frac{x-\rho(0)}{T}\right|^2dx- \|\eta[k]\|^2_{L^2}\\
&=& \frac{1}{6T^2}|\rho(0)|^2- \|\eta[k]\|^2_{L^2}.\end{array}
                                                  \end{eqnarray*}
Therefore
\begin{equation}\label{ine3}
 |\rho(0)|\leq \left(18T^2\|\eta[k]\|^2_{L^2}\right)^{1/3}=\rho_0(\mbox{say}).
\end{equation}
Since $0\geq \rho(x)\geq \rho(0),$ hence from Lemma \ref{BClemma} and \eqref{ine3}, we have 
\begin{enumerate}
 \item[(i).] If $t(x)\leq T/2$, then $T-t(x)\geq T/2$, which implies $ \frac{x}{T-t(x)}\leq \frac{2x}{T}\leq \frac{2R_0}{T}$.
 \item[(ii).] If $t(x)\geq T/2$, then 
 $h_+\left(\frac{x}{T-t(x)}\right)=-\frac{\rho(x)}{t(x)}\leq \frac{2|\rho(0)|}{T}$
\end{enumerate}
and hence there exists a $\Lambda>0$ such that 
 $\frac{x}{T-t(x)}\leq \Lambda.$
Define $\tilde{y}$ by 
  \begin{eqnarray}
  \tilde{y}(x)&=&
\left\{\begin{array}{lllll}
y(x) &\mbox{if}& x\in[-C,0],\\
x &\mbox{if}& y(-C)\geq -C \ \mbox{and}\ x<-C,\\
y(-C) &\mbox{if}& y(-C)\leq-C \ \mbox{and}\ x\in[y(-C),-C],\\
x &\mbox{if}& x\leq y(-C).
   \end{array}\right.\end{eqnarray}
   Then 
   \begin{eqnarray*}
   \tilde{y}(x)=
y(x) \mathbf{1}_{[-C,0]}+ \mbox{min}\{-C,y(-C)\}\mathbf{1}_{[\mbox{min}\{-C,y(-C)\},-C)}+ x \mathbf{1}_{\{x<\mbox{min}\{-C,y(-C)\}\}}.
   \end{eqnarray*}
  and therefore 
  \begin{eqnarray*}
  \tilde{y}(x)= y(x) \mathbf{1}_{[R(T),R_0]}+
\mbox{max}\{y(R_0), R_0\}\mathbf{1}_{(R_0,\mbox{max}\{R_0,y(R_0)\}]}+
x \mathbf{1}_{\{x>\mbox{max}\{R_0,y(R_0)\}\}}.  
  \end{eqnarray*}
  Hence if $y(-C)<-C,$ then for $x\in[y(-x),-C,$ $y(x)\leq y(-C)=\tilde{y}(x)$ which implies 
  $\frac{x-y(x)}{T}\geq \frac{x-\tilde{y}(x)}{T}=\frac{x-y(-C)}{T}>0.$
  Therefore 
  \begin{eqnarray*}
  \begin{array}{lllll}
\int\limits_{-\infty}^{0}\left|\frac{x-y}{T}-\eta[k]\right|^2&=&  \int\limits_{-\infty}^{y(-C)}\left|\frac{x-y}{T}\right|^2+
\int\limits_{y(-C)}^{-C}\left|\frac{x-y}{T}\right|^2+\int\limits_{-C}^{0}\left|\frac{x-y}{T}-\eta[k]\right|^2\\
&\geq& \int\limits_{-\infty}^{y(-C)}\left|\frac{x-\tilde{y}}{T}\right|^2+
\int\limits_{y(-C)}^{-C}\left|\frac{x-y(-C)}{T}\right|^2+\int\limits_{-C}^{0}\left|\frac{x-y}{T}-\eta[k]\right|^2\\
&\geq& \frac{(-C-y(-C))^3}{3T^2},\end{array}
                                   \end{eqnarray*}
and if $y(R_0)>R_0$, then 
\begin{eqnarray*}
\begin{array}{llll}
 \int\limits_{R(T)}^{\infty}\left|\frac{x-y}{T}-\eta[k]\right|^2\geq \int\limits_{R_0}^{\infty}\left|\frac{x-y}{T}\right|^2
\geq \int\limits_{R_0}^{y(R_0)}\left|\frac{x-\tilde{y}}{T}\right|^2=\frac{(y(R_0)-R_0)^3}{3T^2}.\end{array}
\end{eqnarray*}
Since $\tilde{J}(R(T),y,\rho)\leq 2\|\eta[k]\|^2_{L^2}$, hence 
$
 |-C-y(-C)|^3\leq 6T^2\|\eta[k]\|^2_{L^2}$ and 
$|y(R_0)-R_0|^3\leq 6T^2\|\eta[k]\|^2_{L^2}.$ Therefore $ y(-C)\geq -C-(6T^2\|\eta[k]\|^2_{L^2})^{1/3}$ and $y(R_0)\leq R_0+ (6T^2\|\eta[k]\|^2_{L^2})^{1/3}.$
Again by construction, 
$\tilde{J}(R(T),\rho,y)\geq \tilde{J}(R(T),\rho,\tilde{y})$.
Let us denote $M_1= R_0+ (6T^2\|\eta[k]\|^2_{L^2})^{1/3}$.
Then we obtain the following
\begin{lemma}
 Let $R_0,\rho_0,M_1$ be defined as above and define a new class of admissible set $\bar{A}_1$, by 
$$
 \bar{A}_1=\{(R(T),\rho,y): \  0\leq R(T)\leq R_0, \ \rho_0\leq \rho\leq 0, \  y(x)=x \ \mbox{if}\ x\notin[-M_1,M_1]\}.$$
Then 
\begin{eqnarray*}
 \inf_{\bar{A}_1} \tilde{J}=\inf_{A_1} \tilde{J}.
\end{eqnarray*}
\end{lemma}
\begin{lemma}\label{newset}
 There exists $(\bar{R}(T), \bar{\rho},\bar{y})\in \bar{A}_1$ such that 
 \begin{eqnarray}
 \inf_{(R(T),\rho,y)\in\bar{A}_1} \tilde{J}(R(T),\rho,y)=\tilde{J}(\bar{R}(T), \bar{\rho},\bar{y}).
\end{eqnarray}
\end{lemma}
\begin{proof}
 Proof is trivial due to Helly's theorem. 
\end{proof}
\subsection{Proof of Theorem \ref{maintheoremop}}
\begin{proof}
 Let $\bar{R}(T), \bar{\rho},\bar{y}$ be as in Lemma \ref{newset}, then the desired initial data can be constructed from Lemmas \ref{BClemma}
 and \ref{BClemma2}.
\end{proof}

\section{Exact control problem for discontinuous flux}\label{section:exactcontrol}

 \noindent \textbf{Reachable set}: Let $T,C_1,C_2,B_1,B_2, R\in \mathbb{R}$ be given so that $T>0, C_1<0<C_2$, $B_1<0<B_2$.
 Let $\delta>0$ be an arbitrary small number such that  $B_1+\delta<0, B_2-\delta>0$.  Then in order to define $Reachable\ set$ we
 need to consider the following 2 cases.

\noindent\textbf{Case 1:} $R\in (0,C_2).$\\
\noindent Let us consider any  non decreasing functions  $y:[C_1,0]\cup[R,C_2]\rightarrow[B_1+\delta,B_2-\delta]$ and  
 $\rho:[0,R]\rightarrow [B_1-\delta,0]$ which satisfies 
 $xy(x)\geq 0$ for all $x\in [C_1,0]\cup[R,C_2]$ and $y(x)\leq \rho(0)\ \mbox{for all}\ x\in[C_1,0]$.
Then by Lemma \ref{BClemma} there  exists a unique non increasing function
 $t:[0,R]\rightarrow [0,T],$ which satisfies
$-\frac{\rho(x)}{t(x)}=h_+\left(\frac{x}{T-t(x)}\right) \ a.e.x\ \in(0,R).$

Let $\rho(\cdot), y(\cdot), t (\cdot)$ be  as above then
we define a function $ W:[C_1,C_2]\rightarrow \mathbb{R}$   by 
 \begin{eqnarray}\label{Reachable1}
 \begin{array}{ll}
 W(x)=(g^\prime)^{-1}\left(\frac{x-y(x)}{T}\right)\mathbf{1}_{[C_1,0]} &+ (f^\prime)^{-1}\left(\frac{x}{T-t(x)}\right)\mathbf{1}_{[0,R]}\\
  &+(f^\prime)^{-1}\left(\frac{x-y(x)}{T}\right)\mathbf{1}_{[R,C_2]}.
 \end{array}
  \end{eqnarray}
Then  we define the  $Reachable \ set$ associated $ T,\delta, R, C_1,C_2,B_1,B_2$  by 
$$ Reachable \ set_+=\{ W:[C_1,C_2]\rightarrow \mathbb{R} \ \mbox{satisfies} \ \eqref{Reachable1}\}.$$
\noindent\textbf{Case 2:} $R\in (C_1,0).$\\
  Let us consider any  non decreasing functions  $y:[C_1,R]\cup[0,C_2]\rightarrow[B_1+\delta,B_2-\delta]$ and  
 $\rho:[R,0]\rightarrow [0, B_2+\delta]$ which satisfies 
 $xy(x)\geq 0$ for all $x\in [C_1,R]\cup[0,C_2]$ and $ \rho(0)\leq y(x),\ \mbox{for all}\ x\in[0,C_2]$.
Then by Lemma \ref{BClemma} there  exists a unique non increasing function
 $t:[R,0]\rightarrow [0,T],$ which satisfies
$-\frac{\rho(x)}{t(x)}=h_-\left(\frac{x}{T-t(x)}\right) \ a.e.x\ \in(R,0).$\\
Let $\rho(\cdot), y(\cdot), t (\cdot)$ be  as above then we define a function $ W:[C_1,C_2]\rightarrow \mathbb{R}$   by
\begin{eqnarray}\label{Reachable2}
 \begin{array}{ll}
 W(x)=(g^\prime)^{-1}\left(\frac{x-y(x)}{T}\right)\mathbf{1}_{[C_1,R]} &+ (g^\prime)^{-1}\left(\frac{x}{T-t(x)}\right)\mathbf{1}_{[R,0]}\\
  &+(f^\prime)^{-1}\left(\frac{x-y(x)}{T}\right)\mathbf{1}_{[0,C_2]}.
 \end{array}
  \end{eqnarray}
Then  we define the  $Reachable \ set$ associated $ T,\delta, R, C_1,C_2,B_1,B_2$  by 
$$ Reachable \ set_-=\{ W:[C_1,C_2]\rightarrow \mathbb{R} \ \mbox{satisfies} \ \eqref{Reachable2}\}.$$
Finally clubbing Case 1 and Case 2, we define 
$$ Reachable \ set= Reachable \ set_+\cup  Reachable \ set_-.$$
 \begin{theorem}\label{exact-control}
   Let $T>0,C_1<0<C_2, B_1<0<B_2$. Assume that  $\bar{u}_0 \in L^\infty(\mathbb{R}\setminus(B_1,B_2))$ and $W\in Reachable\ set$. 
   Then there exist a solution $u\in L^\infty(\mathbb{R}\times (0,T))$  of (\ref{conlaw-equation})
 such that 
\begin{eqnarray} u (x,T) &=& W (x) \quad x \in (C_1,C_2), \\
 u (x,0) &=& u_0 (x)\mathbf{1}_{\mathbb{R}\setminus(B_1,B_2)} + \bar{u}_0 (x)\mathbf{1}_{ (B_1,B_2)}
 \end{eqnarray}
 \end{theorem}
In order to prove the above Theorem, we need the following free region Lemmas and the backward construction Lemmas  \ref{BClemma},\ref{BClemma2}.
\begin{lemma}\label{FreeRegion1}
 Let $0<B_2, 0<C_2$. Let us assume that $u_0\in L^\infty(B_2,\infty)$ be given. Let $P_2>C_2$ be any number, then there exists $\lambda_2>0$ and a
 solution  $u \in L^\infty(\mathbb{R}_+\times [0,T])$ of the following  system
 \begin{eqnarray}
  \begin{array}{llll}
    u_t+f(u)_x&=&0 \ \ \ \mbox{if}\ (x,t)\in\mathbb{R}_+\times(0,T),\\
   u(x,t)&=&\lambda_2 \ \mbox{if}\ (x,t)\in Q_2,\\
   u(x,0)&=&u_0 \ \mbox{if}\ x\in(B_2,\infty),
  \end{array}
 \end{eqnarray}
 where the domain $Q_2$, is given by 
 $Q_2=\{ (x,t) : 0\leq t \leq T, 0\leq x \leq (t-T)\frac{P_2-B_2}{T}+P_2 \}.$
\end{lemma}
\begin{proof} One can choose $u_0(x)=\lambda_2$, for $x\in(0,B_2)$, where $\lambda_2$ is some  large positive number. Roughly speaking, the 
superlinear growth of $f$ allows a large shock due to $\lambda_2$, which kills the given $u_0$ in $(B_2,\infty)$.
The rigorous  proof follows as in the 
same spirit of the free region Lemmas 2.2, 2.3, 2.4 as in \cite{Sco}.
\end{proof}
\begin{lemma}\label{FreeRegion2}
 Let $B_1<0, C_1<0$. Let us assume that $u_0\in L^\infty(-\infty,B_1)$ be given.
 Let $P_1<C_1$ be any number, then there exists $\lambda_1<0$ and a
 solution  $u \in L^\infty(\mathbb{R}_-\times [0,T])$ of the following  system
  \begin{eqnarray}
  \begin{array}{llll}
    u_t+g(u)_x&=&0 \ \ \ \mbox{if}\ (x,t)\in\mathbb{R}_-\times(0,T),\\
  u(x,t)&=&\lambda_1 \ \mbox{if}\ (x,t)\in Q_1,\\
   u(x,0)&=&u_0 \ \mbox{if}\ x\in(-\infty,B_1),
  \end{array}
 \end{eqnarray}
 where the domain $Q_1$, is given by 
 $Q_1=\{ (x,t) : 0\leq t \leq T, 0\geq x \geq (t-T)\frac{P_1-B_1}{T}+P_1 \}.$
\end{lemma}
\begin{proof}
 Similarly by choosing $u_0(x)=\lambda_1$, for $x\in(B_1,0)$, where $\lambda_2$ is some  large negative number. 
 Proof is similar like as in the previous lemma.
\end{proof} 
\begin{figure}[ht]\label{Fig7}
        \centering
        \def\svgwidth{0.8\textwidth}
        \begingroup
    \setlength{\unitlength}{\svgwidth}
  \begin{picture}(1,0.55363752)%
    \put(0,0.1){\includegraphics[width=0.9\textwidth]{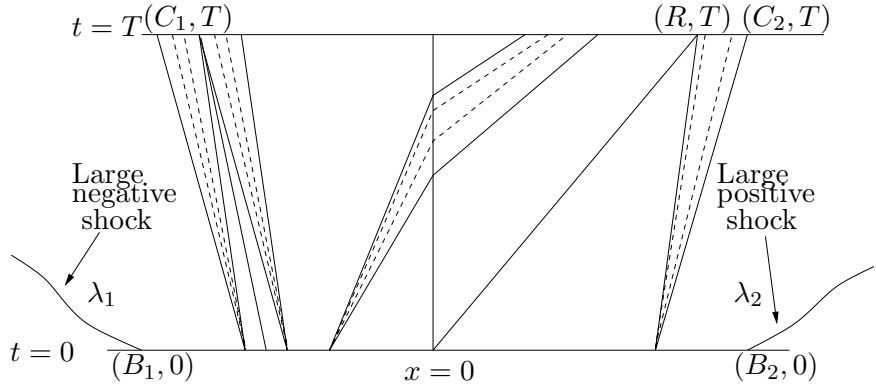}}%
     \put(0.918,0.32){\color[rgb]{0,0,0}\makebox(0,0)[lb]{\smash{Large}}}%
        \put(0.918,0.295){\color[rgb]{0,0,0}\makebox(0,0)[lb]{\smash{positive}}}%
    \put(0.918,0.26){\color[rgb]{0,0,0}\makebox(0,0)[lb]{\smash{ shock}}}
    \put(0.08,0.32){\color[rgb]{0,0,0}\makebox(0,0)[lb]{\smash{Large}}}%
        \put(0.08,0.295){\color[rgb]{0,0,0}\makebox(0,0)[lb]{\smash{negative}}}%
    \put(0.08,0.26){\color[rgb]{0,0,0}\makebox(0,0)[lb]{\smash{ shock}}}%
        \put(0.0,0.087249){\color[rgb]{0,0,0}\makebox(0,0)[lb]{\smash{$t=0$}}}
                \put(0.835,0.52587249){\color[rgb]{0,0,0}\makebox(0,0)[lb]{\smash{$(R,T)$}}}%
                                \put(0.1723835,0.52587249){\color[rgb]{0,0,0}\makebox(0,0)[lb]{\smash{$(C_1,T)$}}}%
                \put(0.95,0.52587249){\color[rgb]{0,0,0}\makebox(0,0)[lb]{\smash{$(C_2,T)$}}}%
        \put(0.512,0.0629){\color[rgb]{0,0,0}\makebox(0,0)[lb]{\smash{$x=0$}}}%
    \put(0.08,0.51249){\color[rgb]{0,0,0}\makebox(0,0)[lb]{\smash{$t=T$}}}%
        \put(0.1,0.1649){\color[rgb]{0,0,0}\makebox(0,0)[lb]{\smash{$\lambda_1$}}}%
                \put(0.94,0.1649){\color[rgb]{0,0,0}\makebox(0,0)[lb]{\smash{$\lambda_2$}}}%
                                \put(0.94,0.07){\color[rgb]{0,0,0}\makebox(0,0)[lb]{\smash{$(B_2, 0)$}}}%
                \put(0.13,0.07){\color[rgb]{0,0,0}\makebox(0,0)[lb]{\smash{$(B_1,0)$}}}%
  \end{picture}%
\endgroup
        \caption{An illustration of the Theorem}
        \label{fig:8}
\end{figure}
\begin{proof}[Proof of Theorem \ref{exact-control}.] Let $\delta>0$ be an arbitrary small number. Then define an initial data in the 
domain $(B_1,B_1,+\delta)\cup (B_2-\delta, B_2)$ by $u_0(x)=\lambda_1\mathbf{1}_{(B_1,B_1,+\delta)}+\lambda_2\mathbf{1}_{(B_2-\delta, B_2)}$,
where $\lambda_1$ and $\lambda_2$ are as in Lemma \ref{FreeRegion2} and Lemma \ref{FreeRegion1} respectively.
From the  above two Lemmas it is clear  that in the region $Q_1\cup Q_2$
there is no influence of the given initial data $u_0\in \mathbb{R}\setminus (B_1,B_2),$ which allow us to use the backward construction Lemma 
\ref{BClemma}, \ref{BClemma2} in the domain
$Q_1\cup Q_2$.
Let us consider
Case 1, i.e., consider any  $R\in (0,C_2)$. Then given $\rho(\cdot),t(\cdot),y(\cdot)$, we apply Lemma \ref{BClemma}.
 Therefore given any $W(x)\in Reachable\ set _+$, we obtain a solution $u\in L^\infty(\mathbb{R}\times (0,T))$  of (\ref{conlaw-equation}) such that
 $u(x,T)=W(x)$ for $x\in (C_1,C_2)$. Similarly one  can construct a solution by using Lemma \ref{BClemma2} 
 when  $W(x)\in Reachable\ set _-$. Hence the theorem.
\end{proof}
\begin{remark}
Due to the explicit formulas (\ref{r2.12}), (\ref{r2.13}) in Theorem \ref{AG1}, the reachable set in Theorem  \ref{exact-control} is optimal.
\end{remark}
\section*{Acknowledgments} The first author would like to thank Gran Sasso Science Institute, L'Aquila, Italy for the hospitality during his
visit  and also IFCAM for the funding.


\begin{thebibliography}{0}

\bibitem{Sco}
\newblock Adimurthi, S. S. Ghoshal and G. D. Veerappa Gowda,
\newblock Exact  controllability of scalar conservation law with strict convex flux, \emph{Math. Control Relat. Fields.},
 \textbf{04},   (2014), 401--449.

\bibitem{Sop}
\newblock Adimurthi, S. S. Ghoshal and  G.D. Veerappa Gowda,
\newblock Optimal 
controllability for  scalar conservation law with  convex flux, 
 \emph{J. Hyperbolic Differ. Equ.},  \textbf{11},   (2014),  477--491.


\bibitem{Ssh}
\newblock Adimurthi, S. S. Ghoshal and G. D. Veerappa Gowda,
\newblock Structure of an entropy solution of a scalar
conservation law with strict convex flux, \emph{J. Hyper. Differential Equations},  \textbf{09},  (2012), 571--611. 

\bibitem{Kyoto} Adimurthi and G. D. Veerappa Gowda, Conservation laws with discontinuous flux,
 \textit{J. Math. Kyoto Univ.} 43, 1,  (2003), 27--70.
 
 \bibitem{Siam} Adimurthi, J. Jaffre and G.D. Veerappa Gowda, Godunov type methods for scalar
 conservation laws with flux function discontinuous in the space variable, \textit{SIAM J. Numer. Anal.} 
42, 1,  (2004), 179--208. 
 
\bibitem{Jhde} Adimurthi, S. Mishra and G. D. Veerappa Gowda, Optimal entropy solutions for conservation laws
 with discontinuous flux-functions, \textit{J. Hyperbolic Differ. Equ.} 2, 4, (2005), 783--837.


\bibitem{Jde} Adimurthi, S. Mishra and G. D. Veerappa Gowda, Explicit Hopf-Lax type formulas for Hamilton-Jacobi
 equations and conservation laws with discontinuous coefficients, \textit{J. Differential Equations}, 241, (2007),
 1, 1--31.
 
\bibitem{SamJEE}
\newblock B. Andreianov, C. Donadello, S.S. Ghoshal and U. Razafison,
\newblock On the attainability set for triangular
type system of conservation laws with initial data control, \emph{J. Evol. Equ.},  \textbf{15}, (2015),  503--532.
 
\bibitem{And1} B. Andreianov, K. H. Karlsen and  N. H. Risebro, A theory of $L^1$-dissipative solvers for
 scalar conservation laws with discontinuous flux. \textit{Arch. Ration. Mech. Anal.} 201, 1, (2011), 27--86.

 

\bibitem{An1}
     \newblock F. Ancona and A. Marson, 
     \newblock On the attainability set for scalar non linear conservation laws  with boundary control,
     \newblock \textit{SIAM J.Control Optim}, 36, 1,  (1998), 290--312.

 \bibitem{An2}    F. Ancona and G. M. Coclite, On the attainable set for Temple class systems with boundary
controls. \textit{SIAM J. Control Optim.}, 43, 6,  (2005), 2166--2190.

\bibitem{Bre1} A. Bressan and  G. M. Coclite. On the boundary control of systems of conservation laws. \textit{SIAM
J. Control Optim.}, 41, 2,  (2002), 607--622.
     
\bibitem{Burger} R. B\"{u}rger, K.H. Karlsen, N.H. Risebro and J. D. Towers, Well-posedness in $\textrm{BV}_t$ 
and convergence of a difference scheme for continuous sedimentation in ideal clarifier thickener units, 
\textit{ Numer. Math.} 97, 1,  (2004), 25--65.


 \bibitem{BurKarRisTow} R. B\"{u}rger, K.H. Karlsen, N.H. Risebro and J. D. Towers, Monotone difference approximations
 for the simulation of clarifier-thickener units. \textit{Comput. Visual. Sci.}, 6, (2004), 83--91.
     
\bibitem{Cas} C.Castro, F.Palacios and  E.Zuazua, Optimal control and vanishing viscosity for the Burgers equations,
 {\it Integral methods in science and engineering}, 2,  Birkhouser Boston Inc, Boston MA, (2010), 65--90,

\bibitem{Cas2} C.Castro and E.Zuazua, Flux identification for 1-d scalar conservation laws in the presence of shocks,
 {\it Math.Comp.}, 80, (2011), 2025--2070.
 
\bibitem{Coron} J.-M. Coron, Global asymptotic stabilization for controllable systems without drift, \textit{
Math. Control Signals Systems.}, 5, 3,  (1992), 295--312. 

\bibitem{CoronShyam}
\newblock J.-M. Coron, S. Ervedoza, S. S. Ghoshal,  O. Glass and V. Perrollaz, 
\newblock Dissipative boundary conditions for $2\times 2$ hyperbolic systems of conservation laws for entropy solutions in BV- \emph{submitted}.



\bibitem{Da1} C. M. Dafermos, Hyperbolic Conservation Laws in Continuum Physics, 2 nd edition, Springer
Verlag, Berlin, (2000).

\bibitem{Diehl5} S. Diehl, Continuous sedimentation of multi-component particles, \textit{ Math. Methods
Appl. Sci.}, 20, (1997), 1345--1364.

 
\bibitem{Gimseresebro} T. Gimse and N. H. Risebro, Solution of the Cauchy problem for a conservation
 law with a discontinuous flux function, \textit{SIAM J. Math. Anal.}, 23 (1992),  635--648.

 \bibitem{Glass}
\newblock  O. Glass and S. Guerrero, 
     \newblock On the uniform controllability of the Burgers equation,
     \newblock \emph{SIAM J. Control optim.}, \textbf{46}, no.4 (2007), 1211--1238.

\bibitem{Hor}
\newblock T. Horsin, 
     \newblock On the controllability of the Burger equation,
     \newblock \emph{ESIAM, Control optimization and Calculus of variations}, 3,  (1998), 83--95.


\bibitem{Jos}  K. T. Joseph and G. D. Veerappa Gowda, Explicit formula for the solution of Convex  conservation laws
 with boundary condition, {\it Duke Math.J.}, 62, (1991) 401--416.

\bibitem{Towers} J.D. Towers, Convergence of a difference scheme for conservation laws with a 
discontinuous flux, \textit{SIAM J. Numer. Anal.} 38, 2, (2000), 681--698.


 \end{thebibliography}
\end{document}